\documentclass[11pt]{article}

\usepackage{amssymb}
\usepackage{amsmath}
\usepackage{amsthm}
\usepackage{bbm}
\usepackage{bm}
\usepackage{color}
\usepackage{tikz}

\usepackage[margin=1in]{geometry}

\usepackage[backend=bibtex]{biblatex}
\newtheorem{thm}{Theorem}[section]
\newtheorem{lem}[thm]{Lemma}
\newtheorem{definition}[thm]{Definition}
\newtheorem{cor}[thm]{Corollary}

\newtheorem{rem}[thm]{Remark}

\newcommand\numberthis{\addtocounter{equation}{1}\tag{\theequation}}

\newcommand{\bd}[1]{\mathbf{\boldsymbol{#1}}}
\newcommand{\R}{\mathbb{R}}
\newcommand{\Jn}{\bd{J}^\bd{n}}

\DeclareMathOperator{\argmin}{argmin}
\DeclareMathOperator{\argmax}{argmax}

\newcommand{\vertiii}[1]{{\left\vert\kern-0.25ex\left\vert\kern-0.25ex\left\vert #1 
		\right\vert\kern-0.25ex\right\vert\kern-0.25ex\right\vert}}
	
\title{The Hanson-Wright Inequality for Random Tensors}
	
\newcommand{\footremember}[2]{
	\footnote{#2}
	\newcounter{#1}
	\setcounter{#1}{\value{footnote}}
}
 
\author{
	Stefan Bamberger\footremember{tum}{stefan.bamberger@tum.de, Department of Mathematics, Technical University of Munich}
	\and Felix Krahmer\footremember{tum2}{felix.krahmer@tum.de, Department of Mathematics, Technical University of Munich}
	\and Rachel Ward\footremember{uta}{rward@math.utexas.edu, Department of Mathematics, University of Texas at Austin}
}

\begin{document}
	
	\maketitle
		
		\begin{abstract}
			We provide moment bounds for expressions of the type $(X^{(1)} \otimes \dots \otimes X^{(d)})^T A (X^{(1)} \otimes \dots \otimes X^{(d)})$ where $\otimes$ denotes the Kronecker product and $X^{(1)}, \dots, X^{(d)}$ are random vectors with independent, mean $0$, variance $1$, subgaussian entries. The bounds are tight up to constants depending on $d$ for the case of Gaussian random vectors. Our proof also provides a decoupling inequality for expressions of this type. Using these bounds, we obtain new, improved concentration inequalities for expressions of the form $\|B (X^{(1)} \otimes \dots \otimes X^{(d)})\|_2$.
		\end{abstract}

	
	\section{Introduction}
	
	\subsection{Background and studied objects}
	
	Given a matrix $A \in \mathbb{R}^{n \times n}$ and a random vector $X \in \mathbb{R}^n$, the Hanson-Wright inequality provides a tail bound for the chaos $X^T A X - \mathbb{E} X^T A X$. In the original work \cite{hanson1971}, $X$ was assumed to have independent subgaussian entries whose distributions are symmetric about $0$.
	
	This result has been improved and adapted to various settings in a number of works, for example \cite{rudelson2013} gives a version which holds for vectors with general subgaussian entries without the symmetry assumption of the distribution:
	
	\begin{thm}[Theorem 1.1 from \cite{rudelson2013}] \label{thm:hanson_wright}
		Let $A \in \mathbb{R}^{n \times n}$.   Let $X \in \mathbb{R}^n$ be a random vector with independent entries such that $\mathbb{E} X = 0$ and such that $X$ has a subgaussian norm of at most $K$. Then for every $t \geq 0$,
		\[
		\mathbb{P}(|X^T A X - \mathbb{E} X^T A X| > t) \leq 2 \exp\left[- c \min \left\{ \frac{t^2}{K^4 \|A\|_F^2}, \frac{t}{K^2 \|A\|_{2 \rightarrow 2} } \right\} \right]
		\]
		where $\|A\|_F$ is the Frobenius and $\|A\|_{2 \rightarrow 2}$ the spectral norm of $A$.
	\end{thm}
	
	Today, the Hanson-Wright inequality is an important probabilistic tool and can be found in various textbooks covering the basics of signal processing and probability theory, such as \cite{foucart_rauhut} and \cite{vershynin_2018}. It has found numerous applications, in particular it has been a key ingredient for the construction of fast Johnson-Lindenstrauss embeddings \cite{riptojl}.
	
	For subgaussian $X \in \mathbb{R}^n$, linear expressions $\sum_{k = 1}^n a_k X_k$  can be controlled by Hoeffding's inequality, while quadratic (order $2$) expressions $X^T A X = \sum_{j, k = 1}^n A_{j, k} X_j X_k$ can be controlled by the Hanson-Wright inequality. Thus, it is natural to wonder to what extent such control extends to a higher-order subgaussian chaos of the form	
	\begin{equation} \label{eq:order_d_chaos_coupled}
		\sum_{i_1, \dots, i_d} A_{i_1, \dots, i_d} X_{i_1} \dots X_{i_d}.
	\end{equation}
	
	Expressions of this type for subgaussian vectors have been considered in \cite{adamczak2015concentration} where they are controlled using specific tensor norms of the arrays of all expected partial derivatives of certain degree with respect to the entries in $X$.
	
	In contrast, for \emph{independent} random vectors $X^{(1)}, \dots, X^{(d)}$, the decoupled chaos
	\begin{equation} \label{eq:order_d_chaos}
		\sum_{i_1, i_2, \dots, i_d = 1}^n A_{i_1, \dots, i_d} X^{(1)}_{i_1} \dots X^{(d)}_{i_d},
	\end{equation}
	can be controlled with simpler bounds and has been considered in multiple previous works for numerous different distributions of the random vectors \cite{latala_gauss_chaos,chaos_log_concave,kolesko2015moment}.

	In the course of adapting fast Johnson-Lindenstrauss embeddings to data with Kronecker structure as introduced in \cite{battaglino2018practical} (see also \cite{oblivious_sketching, kronecker_jl}), one encounters expressions of the form $(X^{(1)} \otimes \dots \otimes X^{(d)})^T A (X^{(1)} \otimes \dots \otimes X^{(d)})$ which are somewhat intermediate between \eqref{eq:order_d_chaos_coupled} and \eqref{eq:order_d_chaos}, as they can be expanded as
	\begin{align*}
		\sum_{i_1, \dots, i_{2 d}=1}^n A_{i_1, \dots, i_d, i_{d + 1}, \dots, i_{2 d}} X^{(1)}_{i_1} \dots X^{(d)}_{i_d} X^{(1)}_{i_{d + 1}} \dots X^{(d)}_{i_{2 d}}.
		\numberthis \label{eq:chaos_double_coupled}
	\end{align*}
	
	Such random processes are also closely related to embeddings of random tensors of the form
	\begin{align}
		\|B(X^{(1)} \otimes \dots \otimes X^{(d)})\|_2
		\label{eq:BX_norm_kronecker}
	\end{align}
	which have recently been studied by Vershynin \cite{vershynin_random_tensors}.

	Even though \eqref{eq:chaos_double_coupled} can be cast as a specific case of \eqref{eq:order_d_chaos_coupled} for which \cite{adamczak2015concentration}  provides optimal bounds, these bounds are not straightforward to use in this specific situation since they are given in terms of partial derivatives and not in terms of  the coefficients $A_{i_1, \dots, i_{2 d}}$.
	
	The main results of this paper provide moment estimates for the semi-decoupled chaos process \eqref{eq:chaos_double_coupled} that are easier to use as they are explicitly given in terms of the coefficients $A_{i_1, \dots, i_{2 d}}$.
	Our bounds imply improved estimates for \eqref{eq:BX_norm_kronecker} and lay the foundations for an order-optimal analysis of fast Kronecker-structured Johnson-Lindenstrauss embeddings. We refer the reader to our companion paper \cite{kronecker_jl_preprint} for a discussion of the implications in this regard. We nevertheless expect that our results should find broader use beyond these specific applications.
	

	\subsection{Previous work}

	For the case where $X^{(1)}, \dots, X^{(d)}$ are independent \emph{Gaussian} vectors, the concentration of \eqref{eq:order_d_chaos} has been studied in \cite{latala_gauss_chaos} which provides upper and lower moment bounds which match up to a constant factor depending only on the order $d$. We will obtain our main results for subgaussian vectors by careful reduction to the Gaussian bounds.

	Higher order chaos expressions have also been studied for distributions beyond Gaussian. Specifically, \cite{rademacher_higher_order}, Section 9, considers (\ref{eq:order_d_chaos_coupled}) for the case of Rademacher vectors. However, the bounds are more intricate than in \cite{latala_gauss_chaos} and the coefficient array $\bd{A} = (A_{i_1, \dots, i_{d}})_{i_1, \dots, i_{d}=1}^n$ must satisfy a symmetry condition and be diagonal-free, i.e., $A_{i_1, \dots, i_d} = 0$ if any two of the indices $i_1, \dots, i_d$ coincide.
	
	Upper and lower bounds on the moments of \eqref{eq:order_d_chaos} are shown in \cite{chaos_log_concave} and \cite{kolesko2015moment} for the case of symmetric random variables with logarithmically concave and convex tails, meaning that for a random variable $X \in  \mathbb{R}$, the function $t \mapsto - \log \mathbb{P}(|X| \geq t)$ is convex or concave, respectively. However, for general subgaussian random variables, neither of these has to be the case. In addition, these works only consider the decoupled chaos \eqref{eq:order_d_chaos} and provide a decoupling inequality to control \eqref{eq:order_d_chaos_coupled} for diagonal-free $\bd{A}$.
	
	Upper moment bounds for general polynomials of independent subgaussian random variables are provided in \cite{adamczak2015concentration}. Similar to our work, the authors utilize the decoupling techniques of \cite{gauss_decoupling_2}. Since \eqref{eq:chaos_double_coupled} is a polynomial in the entries of $X^{(1)}, \dots, X^{(d)}$, it can also be controlled using the results from \cite{adamczak2015concentration}. Because the aforementioned work also shows that these moment bounds are tight for the case of Gaussian vectors, one of the main results (Theorem \ref{thm:main}) of our work can also be shown using their results. However, their result bounds the corresponding $L_p$ norms in terms of norms of the array of all $d' \leq 2 d$ expected partial derivatives, meaning that significant additional work would be required to relate these derivatives to the expressions in Theorem \ref{thm:main}. We believe, that our approach is not much longer but more insightful. In addition, it provides the decoupling result Theorem \ref{thm:decoupling_total} which will be of independent interest.
	
	More work on related topics include \cite{meller2016two, meller2019tail} where upper and lower bounds for the case of random variables satisfying the moment condition $\|X\|_{2 p} \leq \alpha \|X\|_p$ are considered for the case of positive variables of order $2$. The recent work \cite{polynomials_alpha_subexp} provides a similar bound to \cite{adamczak2015concentration} for functions of the random variables that are not necessarily polynomials.
	
	The decoupling technique used in many proofs of the standard Hanson-Wright inequality relates $X^T A X$ to $X^T A \bar{X}$ where $\bar{X}$ is an independent copy of $X$. This approach was first introduced in \cite{McConnel_Taqqu_Decoupling}, already in a general higher-dimensional form. The general idea is to upper bound convex functions (e.g.~moments) of \eqref{eq:order_d_chaos_coupled} by the corresponding expressions of \eqref{eq:order_d_chaos}, up to a constant. Beside independent, symmetrically distributed entries of the random vectors, the result also requires the coefficient array to be symmetric and diagonal free.
	
	The subsequent work \cite{gauss_decoupling_1} has also shown the reverse decoupling bound, up to constant factors, proving that through  \eqref{eq:order_d_chaos}, one can also provide lower bounds on the moments of \eqref{eq:order_d_chaos_coupled} with the same assumptions on the coefficient array. However, in some applications it can be interesting to consider non-diagonal-free coefficient arrays. For example, in the scenario of $\|B (X^{(1)} \otimes \dots \otimes X^{(d)})\|_2^2$, the coefficient array $B^T B$ cannot be expected to fulfill the diagonal-free condition in general. The work in \cite{gauss_decoupling_2} lifts the restriction of a diagonal-free coefficient array and bounds the tails of slight modifications of \eqref{eq:order_d_chaos} and \eqref{eq:order_d_chaos_coupled} by each other up to certain constants in the case of Gaussian random variables.
	
	The concentration of the norm \eqref{eq:BX_norm_kronecker} has recently been studied for the subgaussian case in \cite{vershynin_random_tensors}. It is shown that
	\begin{equation} \label{eq:vershynin_bound}
		\mathbb{P}\left( \left| \|B (X^{(1)} \otimes \dots \otimes X^{(d)})\|_2 - \|B\|_F \right| > t \right) \leq 2 \exp\left( - \frac{c t^2}{d n^{d - 1} \|B\|_{2 \rightarrow 2}^2 } \right)
	\end{equation}
	for an absolute constant $c$ and for $0 \leq t \leq 2 n^\frac{d}{2} \|B\|_{2 \rightarrow 2}$. This bound suggests that techniques like the chaos moment bounds in \cite{latala_gauss_chaos} could be applied to this problem, which is what we do in this work and leads to Theorem \ref{thm:Ax_concentration} below.
	
	\subsection{Overview of our contribution}

	The goal of this work is to provide upper and lower bounds for the moments of the deviation of \eqref{eq:chaos_double_coupled} from its expectation for vectors with independent subgaussian entries (Theorem \ref{thm:main} below). Key steps of the proof include a decoupling inequality for expressions of the form \eqref{eq:chaos_double_coupled}, 
	Theorem~\ref{thm:decoupling_total}, and a comparison to Gaussian random vectors. Finally, based on our results for \eqref{eq:chaos_double_coupled}, we provide a concentration inequality for \eqref{eq:BX_norm_kronecker} as stated in Theorem \ref{thm:Ax_concentration} which extends previous results of  \cite{vershynin_random_tensors}.
	
	Possible applications of such results include recent developments in norm-preserving maps for vectors with tensor structure in the context of machine learning methods using the kernel trick \cite{battaglino2018practical, oblivious_sketching, kronecker_jl}. 

	\subsection{Notation}
	
	Our results on $X^T A X$ where $X$ is a Kronecker product of $d$ random vectors will depend crucially on the structure of the coefficient matrix $A$ rearranged as a higher-order (specifically order $2 d$) array.  As such, we must establish sophisticated notation for such arrays and their indices.
	
	Consider a vector of dimensions $\bd{n} = (n_1, n_2, \dots, n_d)$ and a subset $I \subset [d]$. We call a function $\bd{i}: I \rightarrow \mathbb{N}$ a partial index of order $d$ on $I$ if for all $l \in I$, $\bd{i}_l := \bd{i}(l) \in [n_l]$. Assume there is exactly one such function if $I = \emptyset$. If $I = [d]$, then $\bd{i}$ is called an index of order $d$. We denote the set of all partial indices of order $d$ on $I$ as $\Jn(I)$;  the set of all indices of order $d$ is denoted by $\Jn := \Jn([d])$. $\Jn$ can be identified with $[n_1] \times \dots \times [n_d]$.
	
	A function $\bd{B}: \Jn \rightarrow \R$ is called an array of order $d$. Because of the aforementioned identification, we also write $\bd{B} \in \R^{n_1 \times \dots \times n_d} =: \R^{\bd{n}}$. For $I \subset [d]$, we define $\R^{\bd{n}}(I)$ to be the set of partial arrays $\bd{B}: \Jn(I) \rightarrow \R$. For $I = [d]$, this is just the aforementioned array definition.
	
	We denote
	\[
	\|\bd{B}\|_2 := \left[ \sum_{\bd{i} \in \Jn(I)} B_{\bd{i}}^2 \right]^\frac{1}{2}
	\]
	for the Frobenius norm of the (partial) array where $B_{\bd{i}} := \bd{B}(\bd{i})$ are its entries.

	For disjoint sets $I,\,J \subset [d]$ and corresponding partial indices $\bd{i} \in \Jn(I)$, $\bd{j} \in \Jn(J)$, define the partial index $\bd{i} \dot\times \bd{j} \in \Jn(I \cup J)$ by
	\begin{equation} \label{eq:operator_dottimes}
		(\bd{i} \dot\times \bd{j})_l = 
		\begin{cases}
			\bd{i}_l & \text{if } l \in I \\
			\bd{j}_l & \text{if } l \in J.
		\end{cases}
	\end{equation}
	
	We will often work with arrays of order $2 d$ whose dimensions along the first $d$ axes are the same as the dimensions along the remaining $d$ ones.  We use the notation $\bd{n}^{\times 2} = (n_1, \dots, n_d, n_1, \dots, n_d)$ to denote such arrays.
	
	For sets $I \subset [2 d]$, $J \subset [d]$ such that $I \cap (J + d) = \emptyset$ and for corresponding partial indices $\bd{i} \in \Jn(I)$, $\bd{j} \in \Jn(J)$, define the partial index $\bd{i} \dot+ \bd{j} \in \bd{J}^{\bd{n}^{\times 2}}(I \cup (J + d))$ by
	\begin{equation} \label{eq:operator_dotplus}
		(\bd{i} \dot+ \bd{j})_l = \begin{cases}
			\bd{i}_l & \text{if } l \in I \\
			\bd{j}_{l - d} & \text{if } l \in J + d.
		\end{cases}
	\end{equation}
	
	For $\bd{i} \in \Jn(I)$ and $J \subset I$, define $\bd{i}_J \in \Jn(J)$ to be the restriction of $\bd{i}$ to $J$, i.e., $(\bd{i}_J)_l = \bd{i}_l$ for all $l \in J$.
	
	As suggested by the explanations above, our convention is to use bold letters for higher order arrays (e.g., $\bd{A}$) while their entries are denoted in non-bold letters (e.g., $A_{\bd{i}}$). For some of our results, we will convert matrices into higher-order arrays by rearranging their entries. In these cases, we will denote the matrices in non-bold letters and use the same letter in bold for the array, e.g., $A$ and $\bd{A}$. For the entries, it will be clear from the indices which object is being referred to.
	Besides that, we will also always use bold letters for array indices (e.g., $\bd{i}$), for vectors of array dimensions (e.g. $\bd{n}$), and for the set $\Jn$.
	
	We denote $Id_n \in \R^{n \times n}$ for the identity matrix,  $\|A\|_F$ for the Frobenius norm of a matrix, and $\|A\|_{2 \rightarrow 2}$ for the spectral norm of a matrix.
	
	For a random variable $Y \in \R$, we define $\|Y\|_{L_p} := (\mathbb{E}|Y|^p)^{1 / p}$ and we define the subgaussian norm $\|Y\|_{\psi_2} := \sup_{p \geq 1} \|Y\|_{L_p} / \sqrt{p}$. For a random vector $X \in \R^n$, we define the subgaussian norm $\|X\|_{\psi_2} := \sup_{v \in \R^n, \|v\|_2 = 1} \|\langle X, v \rangle \|_{\psi_2}$, and we call $X$ isotropic if $\mathbb{E} X X^T = Id_n$.
	
	\subsection{Previous relevant results}
	Since our result is based on the bounds given by Latala in \cite{latala_gauss_chaos}, we also consider the following norms which are also used in that result. In our notation, the norms of interest are stated as follows.
	
	\begin{definition} \label{def:tensor_norm}
		For $\bd{n} \in \mathbb{N}^d$ and an array $\bd{B} \in \mathbb{R}^{\bd{n}}$, we define the following norms for any partition $I_1, \dots, I_\kappa$ of $[d]$.
		\begin{align*}
			\|\bd{B}\|_{I_1, \dots, I_\kappa} :=
			\sup_{\substack{\bd{\alpha}^{(1)} \in \mathbb{R}^{\bd{n}}(I_1), \dots, \bd{\alpha}^{(\kappa)} \in \mathbb{R}^{\bd{n}}(I_\kappa), \\ \|\bd{\alpha}^{(1)}\|_2 = \dots = \|\bd{\alpha}^{(\kappa)}\|_2 = 1}} \quad  \sum_{\bd{i} \in \Jn} B_{\bd{i}} \alpha^{(1)}_{\bd{i}_{I_1}} \dots \alpha^{(\kappa)}_{\bd{i}_{I_\kappa}}.
		\end{align*}
	\end{definition}
	For example, when $d=2$, the array $\bd{B}$ is a matrix and $\| \cdot \|_{\{1,2\}}$ coincides with the Frobenius and $\|\cdot\|_{\{1\}, \{2\}}$ with the spectral norm.
	Latala \cite{latala_gauss_chaos} proved the following upper and lower moment bounds for a decoupled Gaussian chaos of arbitrary order.
	\begin{thm}[Theorem 1 in \cite{latala_gauss_chaos}] \label{thm:gauss_chaos_moments}
		Let $\bd{n} \in \mathbb{N}^d$, $\bd{B} \in \mathbb{R}^{\bd{n}}$, $p \geq 2$.
		
		Let $S(d, \kappa)$ denote the set of partitions of $[d]$ into $\kappa$ nonempty disjoint subsets.  Define
		\begin{align} \label{eq:definition_mpB}
			m_p(\bd{B}) := \sum_{\kappa = 1}^d p^{\kappa / 2} \sum_{(I_1, \dots, I_\kappa) \in S(d, \kappa)} \|\bd{B}\|_{I_1, \dots, I_\kappa}.
		\end{align}
		
		Consider independent Gaussian random vectors $g^{(1)} \sim N(0, Id_{\bd{n}_1}), \dots, g^{(d)} \sim N(0, Id_{\bd{n}_d})$. Then
		\begin{align*}
			\frac{1}{C(d)} m_p(\bd{B}) \leq \left\| \sum_{\bd{i} \in \Jn} B_{\bd{i}} \prod_{l \in [d]} g^{(l)}_{\bd{i}_l} \right\|_{L_p} \leq C(d) m_p(\bd{B}),
		\end{align*}
		where $C(d) > 0$ is a constant that only depends on $d$.
	\end{thm}

	\section{Main results}
	The main contribution of our work is the following new tail bound for $\|A (X^{(1)} \otimes \dots \otimes X^{(d)})\|_2$. Note that it contains the deviation of the non-squared norm. This improves upon the previous result by Vershynin \cite{vershynin_random_tensors} as described in \eqref{eq:vershynin_bound}, up to the constant $C(d)$. By comparison, our result provides a strictly stronger bound for matrices with smaller Frobenius norm and holds for all $t \geq 0$.
	
	\begin{thm} \label{thm:Ax_concentration}
		Let $A \in \R^{n_0 \times n^d}$ be a matrix, $X^{(1)}, \dots, X^{(d)} \in \R^{n}$ independent random vectors with independent, mean $0$, variance $1$ entries with subgaussian norm bounded by $L \geq 1$, and let $X := X^{(1)} \otimes \dots \otimes X^{(d)} \in \R^{n^d}$. Then for a constant $C(d)$ depending only on $d$ and for any $t > 0$,

		\begin{align*}
			\mathbb{P}\left( \left| \|A X\|_2 - \|A\|_F \right| > t \right) \leq
			\begin{cases}
				e^2 \exp\left( - C(d) \frac{t^2}{n^{d - 1} \|A\|_{2 \rightarrow 2}^2} \right)
				& \text{if } t \leq n^{\frac{d}{2}} \|A\|_{2 \rightarrow 2} \\
				e^2 \exp\left( - C(d) \left( \frac{t}{\|A\|_{2 \rightarrow 2}} \right)^\frac{2}{d} \right) & \text{if } t \geq n^\frac{d}{2} \|A\|_{2 \rightarrow 2} \\
				e^2 \exp\left( - C(d) \frac{t^2}{ n^{\frac{d - 1}{2}} \|A\|_{F}^2 } \right) & \text{if } n^{\frac{d - 1}{4}} \|A\|_{2 \rightarrow 2} \leq t \leq n^{\frac{d - 1}{4}} \|A\|_{F}.
			\end{cases}
		\end{align*}
	\end{thm}
	
	Note that the third interval intersects the first two intervals. In any interval of intersection, both bounds hold.

	\begin{rem}
		In addition to extending the previous result in \eqref{eq:vershynin_bound} from \cite{vershynin_random_tensors} to all $t \geq 0$, our result provides a strict improvement of that result for matrices with stable rank $(\|A\|_F / \|A\|_{2 \rightarrow 2})^2 \in (1, n^{\frac{d}{2}} ]$. Corollary \ref{cor:moments_non_squared} provides more complicated but provably optimal moment bounds.
	\end{rem}

	This theorem is a consequence of the following result which gives a generalization of the Hanson-Wright inequality (Theorem \ref{thm:hanson_wright}) in terms of upper and lower moment bounds. Note that the operators $\dot\times$ and $\dot+$ are defined in \eqref{eq:operator_dottimes} and \eqref{eq:operator_dotplus}.
	
	\begin{thm} \label{thm:main}
		For $d \geq 1$, let $\bd{n} = (n_1, \dots, n_d)$ be a vector of dimensions, and let $N = n_1 \dots n_d$.
		
		Let $A \in \R^{N \times N}$ and $X^{(1)} \in \mathbb{R}^{n_1}, \dots, X^{(d)} \in \mathbb{R}^{n_d}$ be random vectors with independent, mean $0$, variance $1$ entries with subgaussian norms bounded by $L \geq 1$. Define $X := X^{(1)} \otimes \dots \otimes X^{(d)}$. There exists a constant $C(d)$, depending only on $d$, such that for all $p \geq 2$,
		
		\[
		\left\| X^T A X - \mathbb{E} X^T A X \right\|_{L_p} \leq C(d) m_p.
		\]
		
		The numbers $m_p$ are defined as follows. By rearranging its entries, regard $A$ as an array $\bd{A} \in \R^{\bd{n}^{\times 2}}$ of order $2 d$ such that
		\[
		X^T A X = \sum_{\bd{i}, \bd{i}' \in \Jn} A_{\bd{i} \dot+ \bd{i}'} \prod_{l \in [d]} X^{(l)}_{\bd{i}_l} X^{(l)}_{\bd{i}'_l}.
		\]
		
		For any $I \subset [d]$ and for $I^c = [d] \backslash I$, define $\bd{A}^{(I)} \in \R^{\bd{n}^{\times 2}}(I^c \cup (I^c + d))$ by
		\begin{align} \label{eq:definition_BI}
			A^{(I)}_{\bd{i} \dot+ \bd{i}'} =
			\sum_{\bd{k} \in \Jn(I)} A_{(\bd{i} \dot\times \bd{k}) \dot+ (\bd{i}' \dot\times \bd{k})}
		\end{align}
		for all $\bd{i}, \bd{i}' \in \Jn(I^c)$.
		
		For $T \subset [2 d]$ and $1 \leq \kappa \leq 2d$, denote by $S(T, \kappa)$ the set of partitions of $T$ into $\kappa$ sets. Then for any $p \geq 1$, define
		\begin{align*}
			m_p := L^{2 d} \sum_{\kappa = 1}^{2 d} p^\frac{\kappa}{2}
			\sum_{\substack{I \subset [d] \\ I \neq [d]}} \quad 
			\sum_{(I_1, \dots, I_\kappa) \in S((I^c) \cup (I^c + d), \kappa)} \|\bd{A}^{(I)}\|_{I_1, \dots, I_\kappa}.
		\end{align*}
		
		If in addition, $X^{(1)} \sim N(0, Id_{n_1}), \dots, X^{(d)} \sim N(0, Id_{n_d})$ are normally distributed (i.e.~$L$ is constant), and $\bd{A}$ satisfies the symmetry condition that for all $l \in [d]$ and any $\bd{i}, \bd{i}' \in \Jn([d] \backslash \{l\})$, $\bd{j}, \bd{j}' \in \Jn(\{l\})$,
		\begin{align} \label{eq:symmetry_condition}
			A_{(\bd{i} \dot\times \bd{j}) \dot+ (\bd{i}' \dot\times \bd{j}')} = 
			A_{(\bd{i} \dot\times \bd{j}') \dot+ (\bd{i}' \dot\times \bd{j})},
		\end{align}
		then also the lower bound
		\begin{align*}
			\tilde{C}(d) m_p \leq
			\left\| X^T A X - \mathbb{E} X^T A X \right\|_{L_p}
		\end{align*}
		holds for all $p \geq 2$.  Here, $\tilde{C}(d) > 0$ only depends on $d$.
	\end{thm}
	
	Note that these upper bounds can directly be converted to tail bounds in the style of Theorems \ref{thm:hanson_wright} or \ref{thm:Ax_concentration} using Lemma \ref{lem:moment_tail_bound}. 
	After introducing the required tools, the proof of Theorem \ref{thm:main} will be split up into two parts. We will prove the upper bound in Subsection \ref{sec:proof_upper_bound} and then the lower bound in Subsection \ref{sec:proof_lower_bound}.
	
	\begin{rem}
		The symmetry condition required for the lower bound is not satisfied for all matrices. However, for any matrix $A$, we can find a matrix $\tilde{A}$ satisfying the symmetry condition and such that $X^T A X = X^T \tilde{A} X$ always holds. To do this, in the array notation we can define $\tilde{\bd{A}}$ by transposing $\bd{A}$ along all possible sets of axes and then taking the mean $\tilde{A}_{\bd{i} \dot+ \bd{i}'} =
		\frac{1}{2^d} \sum_{I \subset [d]} A_{(\bd{i}_{I^c} \dot\times \bd{i}'_{I}) \dot+ (\bd{i}_{I} \dot\times \bd{i}'_{I^c})}$ for any $\bd{i}, \bd{i}' \in \Jn$.
		This is a generalization of taking $\tilde{A} = \frac{1}{2} (A + A^T)$ for $d = 1$. Note however, that $\tilde{A}$ might have significantly smaller norms than $A$ which is why the lower moment bounds in Theorem \ref{thm:main} might not hold for $A$ directly.
	\end{rem}
	
	A central part of our argument is the following specialized decoupling result for expressions as in (\ref{eq:chaos_double_coupled}) which might be of independent interest.

	\begin{thm} \label{thm:decoupling_total}
		Let $\bd{n} = (n_1, \dots, n_d) \in \mathbb{N}^d$, $\bd{A} \in \mathbb{R}^{\bd{n}^{\times 2}}$, $X^{(1)} \in \R^{n_1}, \dots, X^{(d)} \in \R^{n_d}$ random vectors with independent mean $0$, variance $1$ entries and $\bar{X}^{(1)}, \dots, \bar{X}^{(d)}$ corresponding independent copies. Then
		\begin{align*}
			& \left\|\sum_{\bd{i}, \bd{i}' \in \Jn} A_{\bd{i} \dot+ \bd{i}'} \prod_{l \in [d]} X^{(l)}_{\bd{i}_l} X^{(l)}_{\bd{i}'_l} - \mathbb{E} \sum_{\bd{i}, \bd{i}' \in \Jn} A_{\bd{i} \dot+ \bd{i}'} \prod_{l \in [d]} X^{(l)}_{\bd{i}_l} X^{(l)}_{\bd{i}'_l} \right\|_{L_p} \\
			& \leq
			\sum_{\substack{I, J \subset [d]: \\ J \subset I,\, I \backslash J \neq [d]}}
			4^{d - |I|} \left\| \sum_{\substack{\bd{i} \in \Jn(J) \\ \bd{j} \in \Jn(I \backslash J) \\ \bd{k}, \bd{k}' \in \Jn(I^c)}} A_{(\bd{i} \dot\times \bd{j} \dot\times \bd{k}) \dot+ (\bd{i} \dot\times \bd{j} \dot\times \bd{k}')} \prod_{l \in J} \left[ (X^{(l)}_{\bd{i}_l})^2 - 1 \right] \prod_{l \in I^c} X^{(l)}_{\bd{k}_l} \bar{X}^{(l)}_{\bd{k}'_l} \right\|_{L_p}
		\end{align*}
	\end{thm}
	
	\begin{rem}
		Consider the special case in Theorem \ref{thm:decoupling_total} of $X^{(1)}, \dots, X^{(d)}$ being Rademacher vectors, i.e., having independent entries that are $\pm 1$ with a probability of $\frac{1}{2}$ each. Then any squared entry is $1$ almost surely. This implies that the factor $\prod_{l \in J} \left[ (X^{(l)}_{\bd{i}_l})^2 - 1 \right]$ is $0$ unless $J = \emptyset$. So on the right hand side of the inequality in Theorem \ref{thm:decoupling_total}, only the terms with $J = \emptyset$ need to be considered.
	\end{rem}

	\section{Main proofs}
	
	\subsection{Preliminaries}
	
	The classical symmetrization theorem for normed spaces, such as Lemma 6.4.2 in \cite{vershynin_hdp}, can be extended to increasing convex functions of norms as the following result from \cite{decoupling_convex} shows.
	
	\begin{lem}[Special case of Lemma A1 in \cite{decoupling_convex}] \label{lem:symmetrization}
		Let $X_1, \dots, X_n$ be independent, mean $0$ real-valued random variables and $p \geq 1$. Let $\xi_1, \dots, \xi_n$ be independent Rademacher variables that are independent of $X_1, \dots, X_n$. Then
		\[
		\frac{1}{2^p} \mathbb{E}\left| \sum_{k = 1}^n \xi_k X_k \right|^p \leq \mathbb{E}\left| \sum_{k = 1}^n X_k \right|^p \leq 2^p \mathbb{E}\left| \sum_{k = 1}^n \xi_k X_k \right|^p
		\]
	\end{lem}
	
	The decoupling theorem for quadratic forms is a well-known result in probability theory and can be found together with its proof for example as Theorem 8.11 in \cite{foucart_rauhut}. A sufficient version for our purpose can be written as follows:
	
	\begin{thm} \label{thm:decoupling_ord_2}
		Let $A \in \mathbb{R}^{n \times n}$ be a matrix, $X \in \mathbb{R}^n$ a vector with independent mean $0$ entries, and $\bar{X}$ and independent copy of $X$. Let $F: \mathbb{R} \rightarrow \mathbb{R}$ be a convex function. Then
		\begin{align*}
			\mathbb{E} F\left(\sum_{\substack{j, k = 1 \\ j \neq k}}^n A_{j k} X_j X_k \right) \leq 
			\mathbb{E} F\left(4 \sum_{j, k = 1}^n A_{j k} X_j \bar{X}_k \right) 
		\end{align*}
	\end{thm}
	
	Also the following elementary result will be used.
	
	\begin{lem} \label{lem:m1_set_powers}
		Let $T$ be a finite set. Then
		\[
		\sum_{S \subset T} (-1)^{|S|} = 
		\begin{cases}
			1 & \text{if } T = \emptyset \\
			0 & \text{otherwise}.
		\end{cases}
		\]
	\end{lem}
	
	\begin{proof}
		For $T = \emptyset$, the statement is clear. Otherwise fix one element $a \in T$ and then
		\begin{align*}
			\sum_{S \subset T} (-1)^{|S|} &= 	\sum_{\substack{S \subset T \\ \text{s.t. } a \in S}} (-1)^{|S|} + \sum_{\substack{S \subset T \\ \text{s.t. } a \notin S}} (-1)^{|S|}
			= 	\sum_{S \subset T \backslash \{a\}} (-1)^{|S| + 1} +	\sum_{S \subset T \backslash \{a\}} (-1)^{|S|} \\
			&= 	\sum_{S \subset T \backslash \{a\}} (-1)^{|S|} \left[ (-1) + 1 \right] = 0.
		\end{align*}
	\end{proof}

	For the norms in Definition \ref{def:tensor_norm}, we need the following property about restricting arrays to some diagonal entries. This can be obtained directly from a repeated application of Lemma 5.2 in \cite{adamczak2015concentration} (where $K = \{l, l + d\}$ for each $l \in I$). Here again, we use the notation of $\dot\times$ and $\dot+$ from \eqref{eq:operator_dottimes} and \eqref{eq:operator_dotplus}.
	\begin{lem} \label{lem:diagonal_array_norm}
		Let $\bd{A} \in \mathbb{R}^{\bd{n}^{\times 2}}$, $I \subset [d]$ and define $\bd{A}^{[I]} \in \R^{\bd{n}^{\times 2}}$ by
		\begin{align*}
			A^{[I]}_{\bd{i} \dot+ \bd{i}'} := \begin{cases}
				A_{\bd{i} \dot+ \bd{i}'} & \text{if } \forall l \in I: \bd{i}_l = \bd{i}'_l \\
				0 & \text{otherwise}.
			\end{cases}
		\end{align*}
		for all $\bd{i}, \bd{i}' \in \Jn$.
		Then for any partition $I_1, \dots, I_\kappa$ of $[2d]$, we have
		\begin{align*}
			\|\bd{A}^{[I]}\|_{I_1, \dots, I_\kappa} \leq \|\bd{A}\|_{I_1, \dots, I_\kappa}.
		\end{align*}
	\end{lem}

	For comparisons between functions of subgaussian and of Gaussian variables, we will use the concept of strong stochastic domination. See, e.g., \cite{random_series} for the following definition and further explanations.
	
	\begin{definition}[Definition 3.2.1 in \cite{random_series}] \label{def:stoch_dominance}
		Let $X, Y \in \R$ be random variables. We say that $X$ is $(\kappa, \lambda)$-strongly dominated by $Y$ ($X \prec_{(\kappa, \lambda)} Y$) if for every $t > 0$,
		\[
		\mathbb{P}(|X| > t) \leq \kappa \mathbb{P}(\lambda |Y| > t).
		\]
	\end{definition}
	
	It can be shown that linear combinations of independent, stochastically dominated random variables are again stochastically dominated which in turn implies the following statement about expectations of convex functions of these linear combinations.
	
	\begin{thm}[Corollary 3.2.1 in \cite{random_series}] \label{thm:dominance_vector}
		Let $X_1, \dots, X_n, Y_1, \dots, Y_n \in \R$ be independent symmetric random variables and $a_1, \dots, a_n \in \R$ fixed coefficients such that $X_i \prec_{(\kappa, \lambda)} Y_i$. Then for any nondecreasing $\varphi: \R^{+} \rightarrow \R^{+}$,
		\[
		\mathbb{E} \varphi \left( \left| \sum_{i = 1}^n a_i X_i \right| \right)
		\leq 2 \lceil \kappa \rceil \mathbb{E} \varphi \left(\lceil \kappa \rceil \lambda \left| \sum_{i = 1}^n a_i Y_i \right| \right).
		\]
	\end{thm}
	
	Statements similar to the following lemma have been used in multiple works to establish a relation between $\left| \|A x\|_2 - a \right|$ and $\left| \|A x\|_2^2 - a^2 \right|$, for example in the proof of Lemma 5.36 in \cite{vershynin_2012}. For completeness, we state it as a separate result with its proof here.
	
	\begin{lem} \label{lem:a2b2_ab_comparison}
		For real numbers $a, b \geq 0$, $b \neq 0$, it holds that
		\[
		\frac{1}{3} \min\left\{ \frac{|a^2 - b^2|}{b}, \sqrt{|a^2 - b^2|} \right\} \leq
		|a - b| \leq \min\left\{\frac{|a^2 - b^2|}{b}, \sqrt{|a^2 - b^2|} \right\}.
		\]
	\end{lem}
	
	\begin{proof}		
		We obtain
		\[
		|a - b| = \frac{|a^2 - b^2|}{|a + b|} \leq \frac{|a^2 - b^2|}{b},
		\]
		and since $a, b \geq 0$, i.e., $|a - b| \leq |a| + |b| = |a + b|$, it follows that $|a - b|^2 \leq |a - b| |a + b| = |a^2 - b^2|$, proving the second inequality.
		
		For the first inequality, first assume the case $a \leq 2 b$. Then $a + b \leq 3 b$ such that
		\[
		\frac{1}{3} \frac{|a^2 - b^2|}{b} \leq \frac{|a^2 - b^2|}{a + b} = |a - b|.
		\]
		In the case that $a \geq 2 b$, i.e., $a - b \geq b \geq 0$, we obtain
		\begin{align*}
			\frac{1}{3} \sqrt{|a^2 - b^2|} & \leq \frac{1}{3} \sqrt{|a + b| |a - b|} \leq
			\frac{1}{3} \sqrt{(|a - b| + 2 b) |a - b|} \\
			& \leq \frac{1}{3} \sqrt{(|a - b| + 2 |a - b|) |a - b|} = \frac{1}{\sqrt{3}} |a - b| \leq |a - b|.
		\end{align*}
	\end{proof}
	
	Relations between moments and tail bounds have also been well-known in the field. For an overview see, e.g., Chapter 7.3 in \cite{foucart_rauhut}. In this spirit, we state and prove the following small tool for the case of mixed tails which we encounter in this work.
	
	\begin{lem}[Moments and tail bounds] \label{lem:moment_tail_bound}
		Let $T$ be a finite set and $X$ an $\R$ valued random variable such that for all $p \geq p_0 \geq 0$,
		\[
		\|X\|_{L_p} \leq \sum_{k = 1}^{d} \min_{l \in T} p^{e_{k, l}} \gamma_{k, l}
		\]
		for values $\gamma_{k, l} > 0$.
		
		Then for all $t > 0$,
		\[
		\mathbb{P}(|X| > t) \leq e^{p_0} \exp\left( - \min_{k \in [d]} \max_{l \in T} \left( \frac{t}{e d \gamma_{k, l}} \right)^\frac{1}{e_{k, l}} \right). \label{eq:concentration}
		\]
	\end{lem}
	
	\begin{proof}
		Fix any $u > 0$. For any $k \in [d]$, define $l'(k) := \argmax_{l \in T} \left( \frac{u}{\gamma_{k, l}} \right)^{\frac{1}{e_{k, l}}}$, then choose $k' := \argmin_{k \in [d]}  \left( \frac{u}{\gamma_{k, l'(k)}} \right)^{ \frac{1}{e_{k, l'(k)}} }$, $p := \left( \frac{u}{\gamma_{k', l'(k')}} \right)^{\frac{1}{e_{k', l'(k')}}}$, such that $p = \min_{k \in [d]} \max_{l \in T} \left( \frac{u}{\gamma_{k, l}} \right)^{\frac{1}{e_{k, l}}}$.
		
		Applying Markov's inequality to $\mathbb{P}(|X| > e d u) \leq \mathbb{P}(|X|^p > (e d u)^p)$, we obtain that this is $\leq e^{p_0} e^{-p}$ in any case and then choose $u = t / (e d)$.

	\end{proof}

	\subsection{Proof of the upper bound}
	
	\subsubsection{Required tools}
	
	\begin{lem} \label{lem:gauss_comparison_linear}
		There is an absolute constant $C$ such that the following holds. Let $X \in \mathbb{R}^n$ be a mean $0$ subgaussian random vector with $\psi_2$ norm $\leq L$. Take a Gaussian vector $g \sim N(0, Id_n)$ and $a \in \mathbb{R}^n$. Then
		\[
		\mathbb{E} \left| \sum_{k = 1}^n a_k X_k \right|^p \leq (C L)^p \left| \sum_{k = 1}^n a_k g_k \right|^p.
		\]
	\end{lem}
	
	\begin{proof}
		By the assumption on $X$, $\sum_{k = 1}^n a_k X_k = \langle a, X \rangle$ is a mean $0$ subgaussian random variable with $\| \langle a, X \rangle\|_{\psi_2} \leq L \|a\|_2$, implying that for any $p \geq 1$,
		\[
		\mathbb{E} | \langle a, X \rangle |^p \leq (C_1 L \|a\|_2)^p p^\frac{p}{2}.
		\]
		
		On the other hand, $\langle a, g \rangle \sim N(0, \|a\|_2^2)$, so by the known absolute moments of the normal distribution and Stirling's approximation,
		\begin{align*}
			\mathbb{E} | \langle a, g \rangle |^p = & \|a\|_2^p \cdot \frac{2^\frac{p}{2}}{\sqrt{\pi}} \Gamma\left( \frac{p + 1}{2} \right)
			\geq
			\|a\|_2^p \frac{2^\frac{p}{2}}{\sqrt{\pi}} \sqrt{2 \pi} \left( \frac{p + 1}{2} \right)^\frac{p}{2} \exp(- \frac{p + 1}{2}) \\
			\geq &
			2^\frac{p}{2} \|a\|_2^p \sqrt\frac{2}{e} \left( \frac{p}{2 e} \right)^\frac{p}{2} \geq
			\sqrt\frac{2}{e} \left( \frac{1}{e} \right)^\frac{p}{2} \|a\|_2^p  p^\frac{p}{2}
			\geq
			\left( \frac{2}{e^2} \right)^\frac{p}{2} \|a\|_2^p  p^\frac{p}{2},
		\end{align*}
		
		implying that $\mathbb{E}|\langle a, X \rangle |^p \leq \left(\frac{C_1 e}{\sqrt{2}} L \right)^p \mathbb{E} |\langle a, g \rangle|^p$.
	\end{proof}
	
	In order to control arbitrary chaoses, we will derive a similar result as Lemma \ref{lem:gauss_comparison_linear} for squared subgaussian and Gaussian variables. To achieve this, we make use of stochastic domination. The following theorem states that this can be used to compare squared subgaussian and Gaussian variables.
	
	\begin{lem} \label{lem:subgauss_dominance}
		There exist absolute constants $\kappa, \lambda > 0$ such that the following holds. Let $X$ be a subgaussian random variable with $\mathbb{E} X^2 = 1$ and $\|X\|_{\psi_2} \leq L$, $L \geq 1$ and $g \sim N(0, 1)$. Let $\xi,\,\xi' \in \{\pm 1\}$ be Rademacher variables that are independent of $X$ and $g$. Then $\xi (X^2 - 1) \prec_{(\kappa, \lambda L^2)} \xi' (g^2 - 1)$ in the sense of Definition \ref{def:stoch_dominance}.
	\end{lem}
	
	\begin{proof}
		For any $t > 0$,
		\begin{align*}
			\mathbb{P}\left( |\xi (X^2 - 1)| > t \right)
			=
			\mathbb{P} \left( X^2 - 1 > t \right) + 
			\mathbb{P} \left( - (X^2 - 1) > t \right)
		\end{align*}
		
		For a constant $c \geq 1$, the first term can be bounded by
		\begin{align*}
			\mathbb{P} \left( X^2 - 1 > t \right) = 
			\mathbb{P} \left( |X| > \sqrt{1 + t} \right)
			\leq 
			\exp\left( 1 - \frac{1 + t}{c^2 L^2} \right) \leq
			e \cdot e^{- \frac{t}{c^2 L^2}}.
		\end{align*}
		
		The second term is $0$ if $t \geq 1$ since $-(X^2 - 1) \leq 1$. For $t \leq 1$, $e^{- \frac{t}{c^2 L^2}} \geq e^{- \frac{1}{c^2 L^2}} \geq e^{- 1}$. Then it holds that $\mathbb{P}( - (X^2 - 1) > t) \leq 1 \leq e \cdot e^{- \frac{t}{c^2 L^2}}$, and altogether we obtain
		\begin{align*}
			\mathbb{P}\left( |\xi (X^2 - 1)| > t \right) \leq 2 e \cdot e^{- \frac{t}{c^2 L^2}}.
		\end{align*}

		On the other hand, for any $\lambda > 0$,
		\begin{align*}
			\mathbb{P} \left(\lambda L^2 |\xi' (g^2 - 1)| > t \right) \geq
			\mathbb{P} \left(g^2 - 1 > \frac{t}{\lambda L^2} \right) \geq
			\mathbb{P} \left(|g| \geq \sqrt{1 + \frac{t}{\lambda L^2}} \right).
		\end{align*}
		
		To bound this, we use the following properties of the normal distribution: (see Proposition 7.5 in \cite{foucart_rauhut})
		
		\begin{align}
			\mathbb{P}(|g| \geq u) \geq \sqrt\frac{2}{\pi} \frac{1}{u} \left( 1 - \frac{1}{u^2} \right) e^{- \frac{u^2}{2} }, \qquad
			\mathbb{P}(|g| \geq u) \geq \left( 1 - \sqrt\frac{2}{\pi} u \right) e^{- \frac{u^2}{2} }. \label{eq:gauss_bound_2}
		\end{align}
		
		For $u \leq \frac{1}{4}$, the second inequality in (\ref{eq:gauss_bound_2}) yields $\mathbb{P} \left(|g| \geq \sqrt{1 + u} \right) \geq \frac{1}{10} e^{- \frac{1 + u}{2} }$.
		
		For $u \geq \frac{1}{4}$, the first inequality in (\ref{eq:gauss_bound_2}) gives $\mathbb{P} \left(|g| \geq \sqrt{1 + u} \right) \geq \frac{1}{5} \sqrt\frac{2}{\pi} \frac{1}{\sqrt{1 + u}} e^{- \frac{1 + u}{2} }$. Using that $\frac{1}{\sqrt{1 + u}} \geq e^{-\frac{1}{2} u}$ for all $u > 0$, we obtain
		\begin{align*}
			\mathbb{P} \left(|g| \geq \sqrt{1 + u} \right) \geq
			\frac{1}{5} \sqrt\frac{2}{\pi} e^{-\frac{1}{2} u} \exp\left( - \frac{1 + u}{2} \right) 
			=
			\frac{1}{5} \sqrt\frac{2}{\pi} \exp\left( - \frac{1}{2} - u \right) \geq
			\frac{1}{11} e^{-u}.
		\end{align*}
		
		So for any $u > 0$, $\mathbb{P}(|g| > \sqrt{1 + u}) \geq \frac{1}{17} e^{-u}$. By choosing $\lambda = c^2$ and combining,
		\begin{align*}
			\mathbb{P}\left( |\xi(X^2 - 1)| > t \right) \leq 2 e \cdot e^{-\frac{t}{\lambda L^2}} \leq 93 \cdot \frac{1}{17} e^{-\frac{t}{\lambda L^2}} \leq 92 \mathbb{P}\left(\lambda L^2 | \xi'(g^2 - 1)| > t\right).
		\end{align*}
	\end{proof}

	\begin{thm} \label{thm:quad_gauss_replacement}
		There is an absolute constant $C > 0$ such that the following holds. Let $X \in \R^n$ have independent entries that have mean $0$ and variance $1$ and are subgaussian with $\psi_2$ norm $\leq L$ for an $L \geq 1$. Take a Gaussian vector $g \sim N(0, Id_n)$ and $a \in \R^n$. Then
		\begin{align*}
			\mathbb{E}\left|\sum_{k = 1}^n a_k (X_k^2 - 1) \right|^p \leq
			(C L^2)^p \mathbb{E}\left|\sum_{k = 1}^n a_k (g_k^2 - 1) \right|^p.
		\end{align*}
	\end{thm}

	\begin{proof}
		Consider independent Rademacher variables $\xi_1, \dots, \xi_n, \bar{\xi}_1, \dots, \bar{\xi}_n \in \{\pm 1\}^n$ that are also independent of $X$ and $g$. By the symmetrization Lemma \ref{lem:symmetrization}, it holds that
		\begin{align}
			\mathbb{E}\left|\sum_{k = 1}^n a_k (X_k^2 - 1) \right|^p \leq 2^p \mathbb{E}\left|\sum_{k = 1}^n a_k \xi_k (X_k^2 - 1) \right|^p \nonumber \\
			\mathbb{E}\left|\sum_{k = 1}^n a_k \bar{\xi}_k (g_k^2 - 1) \right|^p \leq 2^p \mathbb{E}\left|\sum_{k = 1}^n a_k (g_k^2 - 1) \right|^p. \label{eq:app_symm_proof_1}
		\end{align}
		
		Using that $\xi_k (X^2 - 1) \prec_{(\kappa, \lambda L^2)} \bar{\xi}_k (g^2 - 1)$ by Lemma \ref{lem:subgauss_dominance} and that $|\cdot|^p$ is a convex nondecreasing function $\R^{+} \rightarrow \R^{+}$, Theorem \ref{thm:dominance_vector} implies that there is a constant $\tilde{C} > 0$ such that
		\begin{align*}
			\mathbb{E}\left|\sum_{k = 1}^n a_k \xi_k (X_k^2 - 1) \right|^p \leq (\tilde{C} L^2)^p \mathbb{E}\left|\sum_{k = 1}^n a_k \bar{\xi}_k (g_k^2 - 1) \right|^p.
		\end{align*}
	\end{proof}
	
	\begin{thm} \label{thm:rearrange_minus1}
		Let $\bd{n} \in \mathbb{N}^d$, $\bd{A} \in \mathbb{R}^{\bd{n}}$, $X^{(1)} \in \mathbb{R}^{n_1}, \dots, X^{(d)} \in \mathbb{R}^{n_d}$, $I \subset [d]$. Then
		\begin{align*}
			\sum_{\bd{i} \in \Jn} A_{\bd{i}} \prod_{l \in [d]} (X^{(l)}_{\bd{i}_l})^2 = 
			\sum_{I \subset [d]} \sum_{\bd{i} \in \Jn([d] \backslash I)} A^{(I)}_{\bd{i}} \prod_{l \in [d] \backslash I} \left[ (X^{(l)}_{\bd{i}_l})^2 - 1 \right]
		\end{align*}
		where for any $\bd{i} \in \Jn([d] \backslash I)$,
		\begin{align*}
			A_{\bd{i}}^{(I)} = \sum_{\bd{j} \in \Jn(I)} A_{\bd{i} \dot\times \bd{j}}.
		\end{align*}
	\end{thm}
	
	\begin{proof}
		Observing that for any $I \subset [d]$, $\bd{i} \in \Jn(I)$,
		\begin{align*}
			\prod_{l \in [d] \backslash I} \left[ (X^{(l)}_{\bd{i}_l})^2 - 1 \right] =
			\sum_{I' \subset [d] \backslash I} (-1)^{|[d] \backslash (I \cup I')|} \prod_{l \in I'} (X^{(l)}_{\bd{i}_l})^2,
		\end{align*}
		we obtain
		\begin{align*}
			& \sum_{\substack{I \subset [d] \\ \bd{i} \in \Jn([d] \backslash I)}} A^{(I)}_{\bd{i}} \prod_{l \in [d] \backslash I} \left[ (X^{(l)}_{\bd{i}_l})^2 - 1 \right]
			=
			\sum_{\substack{I \subset [d] \\ \bd{i} \in \Jn([d] \backslash I) \\ \bd{j} \in \Jn(I)}} A_{\bd{i} \dot\times \bd{j}} \sum_{I' \subset [d] \backslash I} (-1)^{|[d] \backslash (I \cup I')|} \prod_{l \in I'} (X^{(l)}_{\bd{i}_l})^2 \\
			= &
			\sum_{\substack{I \subset [d] \\ I' \subset [d] \backslash I}} (-1)^{|[d] \backslash (I \cup I')|} \sum_{\substack{\bd{i} \in \Jn([d] \backslash I) \\ \bd{j} \in \Jn(I)}} A_{\bd{i} \dot\times \bd{j}} \prod_{l \in I'} (X^{(l)}_{\bd{i}_l})^2
			=
			\sum_{\substack{I' \subset [d] \\ I \subset [d] \backslash I'}} (-1)^{|[d] \backslash (I \cup I')|} \sum_{\bd{i} \in \Jn}  A_{\bd{i}} \prod_{l \in I'} (X^{(l)}_{\bd{i}_l})^2 \\
			= &
			\sum_{I' \subset [d]} \left[ \left( \sum_{I \subset [d] \backslash I'} (-1)^{|([d] \backslash I') \backslash I|} \right) \cdot \left( \sum_{\bd{i} \in \Jn} A_{\bd{i}} \prod_{l \in I'} (X^{(l)}_{\bd{i}_l})^2 \right) \right].
		\end{align*}
		This implies the claim using Lemma \ref{lem:m1_set_powers}.
	\end{proof}
	
	A key to the proof of the upper moment bound in our main result (Theorem \ref{thm:main}) is the decoupling technique of Theorem \ref{thm:decoupling_total}. With the above auxiliary results, we can give the proof of it here.
	
	\begin{proof}[Proof of Theorem \ref{thm:decoupling_total}]
		\begin{align*}
			b := & \sum_{\bd{i}, \bd{i}' \in \Jn} A_{\bd{i} \dot+ \bd{i}'} \prod_{l \in [d]} X^{(l)}_{\bd{i}_l} X^{(l)}_{\bd{i}'_l} 
			= 
			\sum_{I \subset [d]} \sum_{\substack{\bd{i} \in \Jn(I) \\ \bd{j}, \bd{j}' \in \Jn(I^c) \\ \forall l \in I^c: \bd{j}_l \neq \bd{j}'_l}} A_{(\bd{i} \dot\times \bd{j}) \dot+ (\bd{i} \dot \times \bd{j}')} \prod_{l \in I} (X^{(l)}_{\bd{i}_l})^2 \prod_{l \in I^c} X^{(l)}_{\bd{j}_l} X^{(l)}_{\bd{j}'_l}
		\end{align*}
		since each summand $\bd{i}, \bd{i}'$ is precisely considered in the sum for $I = \{l \in [d]: \bd{i}_l = \bd{i}'_l\}$ and no other $I$.
		
		Now applying Theorem \ref{thm:rearrange_minus1} yields
		\begin{align*}
			b = & \sum_{I \subset [d]} \sum_{\bd{i} \in \Jn(I)} \left( \sum_{\substack{\bd{j}, \bd{j}' \in \Jn(I^c) \\ \forall l \in I^c: \bd{j}_l \neq \bd{j}'_l}} A_{(\bd{i} \dot\times \bd{j}) \dot+ (\bd{i} \dot \times \bd{j}')} \prod_{l \in I^c} X^{(l)}_{\bd{j}_l} X^{(l)}_{\bd{j}'_l} \right) \prod_{l \in I} (X^{(l)}_{\bd{i}_l})^2 \\
			= &
			\sum_{I \subset [d]} \sum_{J \subset I} \sum_{\substack{\bd{i} \in \Jn(J) \\ \bd{k} \in \Jn(I \backslash J)}} \left( \sum_{\substack{\bd{j}, \bd{j}' \in \Jn(I^c) \\ \forall l \in I^c: \bd{j}_l \neq \bd{j}'_l}} A_{(\bd{i} \dot\times \bd{j} \dot\times \bd{k}) \dot+ (\bd{i} \dot \times \bd{j}' \dot\times \bd{k})} \prod_{l \in I^c} X^{(l)}_{\bd{j}_l} X^{(l)}_{\bd{j}'_l} \right) \prod_{l \in J} \left[ (X^{(l)}_{\bd{i}_l})^2 - 1 \right] \\
			= &
			\sum_{\substack{I, J \subset [d]: \\ J \subset I}} \sum_{\substack{\bd{i} \in \Jn(J) \\ \bd{k} \in \Jn(I \backslash J) \\ \bd{j}, \bd{j}' \in \Jn(I^c) \\ \forall l \in I^c: \bd{j}_l \neq \bd{j}'_l}} A_{(\bd{i} \dot\times \bd{j} \dot\times \bd{k}) \dot+ (\bd{i} \dot \times \bd{j}' \dot\times \bd{k})} \prod_{l \in I^c} X^{(l)}_{\bd{j}_l} X^{(l)}_{\bd{j}'_l} \prod_{l \in J} \left[ (X^{(l)}_{\bd{i}_l})^2 - 1 \right]
			=: \sum_{\substack{I, J \subset [d]: \\ J \subset I}} S_{I, J}.
		\end{align*}
		
		Because of
		\begin{align*}
			S_{[d], \emptyset} = \sum_{\bd{k} \in \Jn} A_{\bd{k} \dot+ \bd{k}} = \mathbb{E} \sum_{\bd{i}, \bd{i}' \in \Jn} A_{\bd{i} \dot+ \bd{i}'} \prod_{l \in [d]} X^{(l)}_{\bd{i}_l} X^{(l)}_{\bd{i}'_l}
		\end{align*}
		and the triangle inequality, we obtain
		\begin{align} \label{eq:decoupling_general_1}
			\|b - \mathbb{E} b\|_{L_p} \leq \sum_{\substack{I, J \subset [d]: \\ J \subset I, I \backslash J \neq \emptyset}} \|S_{I, J}\|_{L_p}.
		\end{align}

		For any fixed $l_0 \in I^c$, we obtain that $\|S_{I, J}\|_{L_p} =$
		\begin{align*}
			\left\| \sum_{\substack{\bar{\bd{j}}, \bar{\bd{j}}' \in \Jn(\{l_0\}) \\ \bar{\bd{j}}_{l_0} \neq \bar{\bd{j}}'_{l_0}}} \left( \sum_{\substack{\bd{i} \in \Jn(J) \\ \bd{k} \in \Jn(I \backslash J) \\ \bd{j}, \bd{j}' \in \Jn(I^c \backslash \{l_0\}) \\ \forall l \in I^c: \bd{j}_l \neq \bd{j}'_l}} A_{\substack{(\bd{i} \dot\times \bd{j} \dot\times \bar{\bd{j}} \dot\times \bd{k}) \\ \dot+ (\bd{i} \dot \times \bd{j}' \dot\times \bar{\bd{j}}' \dot\times \bd{k})}} \prod_{l \in I^c} X^{(l)}_{\bd{j}_l} X^{(l)}_{\bd{j}'_l} \prod_{l \in J} \left[ (X^{(l)}_{\bd{i}_l})^2 - 1 \right] \right) X^{(l_0)}_{\bar{\bd{j}}_{l_0}} X^{(l_0)}_{\bar{\bd{j}}'_{l_0}} \right\|_{L_p}.
		\end{align*}
		
		We can apply the decoupling Theorem \ref{thm:decoupling_ord_2} to this for the convex function $|\cdot|^p$ and the expectation conditioned on all variables except $X^{(l_0)}$. This leads to $\|S_{I, J}\|_{L_p} \leq$
		\begin{align*}
			4 \left\| \sum_{\bar{\bd{j}}, \bar{\bd{j}}' \in \Jn(\{l_0\})} \left( \sum_{\substack{\bd{i} \in \Jn(J) \\ \bd{k} \in \Jn(I \backslash J) \\ \bd{j}, \bd{j}' \in \Jn(I^c \backslash \{l_0\}) \\ \forall l \in I^c: \bd{j}_l \neq \bd{j}'_l}} A_{\substack{(\bd{i} \dot\times \bd{j} \dot\times \bar{\bd{j}} \dot\times \bd{k}) \\ \dot+ (\bd{i} \dot \times \bd{j}' \dot\times \bar{\bd{j}}' \dot\times \bd{k})}} \prod_{l \in I^c} X^{(l)}_{\bd{j}_l} X^{(l)}_{\bd{j}'_l} \prod_{l \in J} \left[ (X^{(l)}_{\bd{i}_l})^2 - 1 \right] \right) X^{(l_0)}_{\bar{\bd{j}}_{l_0}} \bar{X}^{(l_0)}_{\bar{\bd{j}}'_{l_0}} \right\|_{L_p}.
		\end{align*}
		
		Repeating this procedure iteratively for all other $l \in I^c$, we obtain
		\begin{align*}
			\|S_{I, J}\|_{L_p}
			\leq
			4^{d - |I|} \left\| \sum_{\substack{\bd{i} \in \Jn(J) \\ \bd{k} \in \Jn(I \backslash J) \\ \bd{j}, \bd{j}' \in \Jn(I^c)}} A_{(\bd{i} \dot\times \bd{j} \dot\times \bd{k}) \dot+ (\bd{i} \dot \times \bd{j}' \dot\times \bd{k})} \prod_{l \in I^c} X^{(l)}_{\bd{j}_l} \bar{X}^{(l)}_{\bd{j}'_l} \prod_{l \in J} \left[ (X^{(l)}_{\bd{i}_l})^2 - 1 \right] \right\|_{L_p}.
		\end{align*}
		Substituting this into (\ref{eq:decoupling_general_1}) completes the proof.
	\end{proof}

	The works in \cite{gauss_decoupling_1} and \cite{gauss_decoupling_2} have investigated polynomials with higher powers of Gaussian variables. Since in our scenario, we only have two occurrences of every vector, thus we can repeatedly apply their result for the case of two coinciding indices. Considering that $H_2(x) = x^2 - 1$ is the Hermite polynomial of degree $2$ and leading coefficient $1$, equation (2.9) in \cite{gauss_decoupling_2} in our setup can be written as follows. Note that as suggested there, the case $p \geq 1$ can also be shown using Jensen's inequality which can be used to show this inequality with coefficient $2$.
	
	\begin{lem} \label{lem:gaussian_decoupling}
		Let $a \in \mathbb{R}^n$, $g, \bar{g} \sim N(0, Id_n)$, $p \geq 1$. Then
		\[
		\left\| \sum_{k = 1}^n a_k (g_k^2 - 1) \right\|_{L_p} \leq
		2 \left\| \sum_{k = 1}^n a_k g_k \bar{g}_k \right\|_{L_p}.
		\]
	\end{lem}

	Combining the previous lemmas, now we can prove the upper bound in the main Theorem \ref{thm:main}.
	
	\subsubsection{Proof of Theorem \ref{thm:main}, upper bound} \label{sec:proof_upper_bound}
	
	\textbf{Step 1: Decoupling}
	
	Let $\alpha := \|X^T A X - \mathbb{E} X^T A X\|_{L_p}$.	By Theorem \ref{thm:decoupling_total},
	\begin{equation}
		\alpha \leq
		\sum_{\substack{J \subset I \subset [d] \\ I \backslash J \neq [d]}} 4^{d - |I|} \left\| \sum_{\substack{\bd{i} \in \Jn(J) \\ \bd{k} \in \Jn(I \backslash J) \\ \bd{j}, \bd{j}' \in \Jn(I^c)}} A_{(\bd{i} \dot\times \bd{j} \dot\times \bd{k}) \dot+ (\bd{i} \dot\times \bd{j}' \dot\times \bd{k})} \prod_{l \in I^c} X^{(l)}_{\bd{j}_l} \bar{X}^{(l)}_{\bd{j}'_l}  \prod_{l \in J} \left[ (X^{(l)}_{\bd{i}_l})^2 - 1 \right]  \right\|_{L_p}.
		\label{eq:before_replacement}
	\end{equation}

	\textbf{Step 2: Replacing the subgaussian factors by Gaussians}
	
	In \eqref{eq:before_replacement}, we can repeatedly apply Lemma \ref{lem:gauss_comparison_linear} to replace all the linear subgaussian factors by Gaussian ones. Afterwards, Theorem \ref{thm:quad_gauss_replacement} allows the same for the quadratic terms. Together, this yields,

	\begin{align}
		\alpha \leq
		\sum_{\substack{J \subset I \subset [d] \\ I \backslash \neq [d]}} (C L)^{|I^c| + |J|} \left\| \sum_{\substack{\bd{i} \in \Jn(J) \\ \bd{k} \in \Jn(I \backslash J) \\ \bd{j}, \bd{j}' \in \Jn(I^c)}} A_{(\bd{i} \dot\times \bd{j} \dot\times \bd{k}) \dot+ (\bd{i} \dot\times \bd{j}' \dot\times \bd{k})} \prod_{l \in I^c} g^{(l)}_{\bd{j}_l} \bar{g}^{(l)}_{\bd{j}'_l}  \prod_{l \in J} \left[ (g^{(l)}_{\bd{i}_l})^2 - 1 \right]  \right\|_{L_p}.
		\label{eq:g2_1}
	\end{align}

	\textbf{Step 3: Decoupling of squared Gaussians}
	In an analogous fashion as in step 3, we can successively replace all the factors $\left[ (g^{(l)}_{\bd{i}_l})^2 - 1 \right]$ in (\ref{eq:g2_1}) by $g^{(l)}_{\bd{i}_l} \bar{g^{(l)}}_{\bd{i}_l}$ using Lemma \ref{lem:gaussian_decoupling}. This leads to
	\begin{align*}
		\alpha \leq &
		\sum_{\substack{J \subset I \subset [d] \\ I \backslash J \neq [d]}} (C L)^{|I^c| + |J|} \left\| \sum_{\substack{\bd{i} \in \Jn(J) \\ \bd{k} \in \Jn(I \backslash J) \\ \bd{j}, \bd{j}' \in \Jn(I^c)}} A_{(\bd{i} \dot\times \bd{j} \dot\times \bd{k}) \dot+ (\bd{i} \dot\times \bd{j}' \dot\times \bd{k})} \prod_{l \in I^c} g^{(l)}_{\bd{j}_l} \bar{g}^{(l)}_{\bd{j}'_l}  \prod_{l \in J} g^{(l)}_{\bd{i}_l} \bar{g}^{(l)}_{\bd{i}_l}  \right\|_{L_p} \\
		= &
		\sum_{\substack{J \subset I \subset [d] \\ I \backslash J \neq [d]}} (C L)^{|I^c| + |J|} \left\| \sum_{\bd{i}, \bd{i}' \in \Jn(I^c \cup J)} A^{(I, J)}_{\bd{i} \dot+ \bd{i}'} \prod_{l \in I^c \cup J} g^{(l)}_{\bd{j}_l} \bar{g}^{(l)}_{\bd{j}'_l} \right\|_{L_p}.
	\end{align*}
	where for all $\bd{i}, \bd{i}' \in \Jn(J \cup I^c)$,
	\begin{align} \label{eq:definition_AIJ}
		A^{(I, J)}_{\bd{i} \dot+ \bd{i}'} =
		\begin{cases}
			\sum_{\bd{k} \in \Jn(I \backslash J)} A_{(\bd{i} \dot\times \bd{k}) \dot+ (\bd{i}' \dot\times \bd{k})} & \text{if } \forall l \in J: \bd{i}_l = \bd{i}_l' \\
			0 & \text{otherwise}.
		\end{cases}
	\end{align}
	
	\textbf{Step 4: Completing the proof}
	Then Theorem \ref{thm:gauss_chaos_moments} yields that
	\begin{align*}
		\left\| \sum_{\bd{i}, \bd{i}' \in \Jn(J \cup I^c)} A^{(I, J)}_{\bd{i} \dot+ \bd{i}'} \prod_{l \in I^c \cup J} g^{(l)}_{\bd{i}_l} \bar{g}^{(l)}_{\bd{i}'_l} \right\|_{L_p}
		\leq \tilde{m}^{(I, J)}_p
	\end{align*}
	where for $S((J \cup I^c) \cup ((J \cup I^c) + d), \kappa)$ being the set of all partitions of $(J \cup I^c) \cup ((J \cup I^c) + d)$ into $\kappa$ sets,
	\begin{align*}
		\tilde{m}^{(I, J)}_p := \sum_{\kappa = 1}^d p^{\kappa / 2} \sum_{(I_1, \dots, I_\kappa) \in S((J \cup I^c) \cup ((J \cup I^c) + d), \kappa)} \|\bd{A}^{(I, J)}\|_{I_1, \dots, I_\kappa}.
	\end{align*}
	
	By Lemma \ref{lem:diagonal_array_norm}, $\|\bd{A}^{(I, J)}\|_{I_1, \dots, I_\kappa} \leq \|\bd{A}^{(I)}\|_{I_1, \dots, I_\kappa}$ where $\bd{A}^{(I)} = \bd{A}^{(I, \emptyset)}$ as given in the statement of Theorem \ref{thm:main}. Together with this, the upper bound in Theorem \ref{thm:main} follows.

	\subsection{Proof of the lower bound}
	
	\subsubsection{Required tools}
	
	In this section, we will prove the lower bound in Theorem \ref{thm:main}. Unlike the upper bound, we will only prove this for the case of Gaussian vectors. Indeed, for arbitrary subgaussian distributions, the lower bound fails to hold as the following simple example for the case $d = 1$ shows: Consider the identity matrix $Id_n$ and a Rademacher vector $\xi \in \{\pm 1\}^n$. Then the object of interest in Theorem \ref{thm:main} is $\xi^T Id_n \xi - \mathbb{E}[\xi^T Id_n \xi] = 0$ even though the moment bounds $m_p$ would be $> 0$.
	
	We follow the approach of reversing all steps in the proof of the upper bound, without the Gaussian comparison steps. This is why also the two decoupling steps before and after the Gaussian comparison can be performed together.
	
	As mentioned before, Gaussian decoupling, with upper as well as lower bounds, has been studied in \cite{gauss_decoupling_2} where central ideas of \cite{gauss_decoupling_1} have been used. \cite{gauss_decoupling_2} provides a decoupling inequality for Gaussian chaos with an arbitrary number of coinciding indices. Similarly to Lemma \ref{lem:gaussian_decoupling}, we can adapt the result of Equation (2.9) in \cite{gauss_decoupling_2} to our situation as follows.
	
	\begin{lem} \label{lem:gauss_dec_order_2}
		Let $A \in \mathbb{R}^{n \times n}$ be a symmetric matrix, $g, \bar{g} \sim N(0, Id_n)$ be independent, and $p \geq 1$.
		\begin{align*}
			\left\| \sum_{j, k \in [n]} A_{j, k} g_j \bar{g}_k \right\|_{L_p} \leq \left\|\sum_{j, k \in [n]} A_{j, k} (g_j g_k - \mathbbm{1}_{j = k}) \right\|_{L_p}.
		\end{align*}
	\end{lem}

	To generalize this to cases of multiple axes, we iteratively apply Lemma \ref{lem:gauss_dec_order_2} to obtain the following corollary.
	
	\begin{cor} \label{cor:gauss_rev_decouple}
		Let $\bd{n} \in \mathbb{N}^d$, $\bd{A} \in \mathbb{R}^{\bd{n}^{\times 2}}$ such that $\bd{A}$ satisfies the symmetry condition that for all $l \in [d]$ and any $\bd{i}, \bd{i}' \in \Jn([d] \backslash \{l\})$, $\bd{j}, \bd{j}' \in \Jn(\{l\})$,
		\begin{align} \label{eq:A_symmetry}
			A_{(\bd{i} \dot\times \bd{j}) \dot+ (\bd{i}' \dot\times \bd{j}')} = 
			A_{(\bd{i} \dot\times \bd{j}') \dot+ (\bd{i}' \dot\times \bd{j})}
		\end{align}

		Let $g^{(1)}, \bar{g}^{(1)} \sim N(0, Id_{\bd{n}_1}), \dots, g^{(d)}, \bar{g}^{(d)} \sim N(0, Id_{\bd{n}_d})$ be independent. Then for any set $I \subset [d]$, $p \geq 1$,
		\begin{align*}
			\left\| \sum_{\substack{\bd{i}, \bd{i}' \in \Jn(I) \\ \bd{j}\in \Jn(I^c)}} A_{(\bd{i} \dot\times \bd{j}) \dot+ (\bd{i}' \dot\times \bd{j})} \prod_{l \in I} g^{(l)}_{\bd{i}_l} \bar{g}^{(l)}_{\bd{i}'_l} \right\|_{L_p}
			\leq
			\left\| \sum_{\substack{\bd{i}, \bd{i}' \in \Jn(I) \\ \bd{j} \in \Jn(I^c)}} A_{(\bd{i} \dot\times \bd{j}) \dot+ (\bd{i}' \dot\times \bd{j})} \prod_{l \in I} \left[ g^{(l)}_{\bd{i}_l} g^{(l)}_{\bd{i}'_l} - \mathbbm{1}_{\bd{i}_l = \bd{i}'_l} \right] \right\|_{L_p} 
		\end{align*}
	\end{cor}

	Independently of the Gaussian decoupling approach, the following two lemmas provide a tool to reverse the application of the rearrangement result Theorem \ref{thm:rearrange_minus1} in the proof of the upper bound.
	
	\begin{lem} \label{lem:rearrange_m1}
		
		Let $\bd{A} \in \mathbb{R}^{\bd{n}^{\times 2}}$ be an array of order $2d$ and $X^{(1)} \in \mathbb{R}^{n_1}, \dots X^{(d)} \in \mathbb{R}^{n_d}$ vectors. Then
		
		\begin{align*}
			\sum_{I \subset [d]} \sum_{\bd{i}, \bd{i}' \in \Jn(I)} \sum_{\bd{j} \in \Jn(I^c)} A_{(\bd{i} \dot\times \bd{j}) \dot+ (\bd{i}' \dot\times \bd{j})} \prod_{l \in I} \left[ X^{(l)}_{\bd{i}_{l}} X^{(l)}_{\bd{i}'_l} - \mathbbm{1}_{\bd{i}_l = \bd{i}'_l} \right] = 
			\sum_{\bd{i}, \bd{i}' \in \Jn} A_{\bd{i} \dot+ \bd{i}'} \prod_{l \in [d]} X^{(l)}_{\bd{i}_l} X^{(l)}_{\bd{i}'_l}.
		\end{align*}
	\end{lem}
	
	\begin{proof}
		Note that
		\begin{align*}
			\prod_{l \in I} \left[ X^{(l)}_{\bd{i}_{l}} X^{(l)}_{\bd{i}'_l} - \mathbbm{1}_{\bd{i}_l = \bd{i}'_l} \right] =
			\sum_{J \subset I} (-\mathbbm{1}_{\bd{i}_l = \bd{i}'_l})^{|I \backslash J|} \prod_{l \in J} X^{(l)}_{\bd{i}_{l}} X^{(l)}_{\bd{i}'_l}.
		\end{align*}
		
		Using this, we obtain
		\begin{align*}
			\alpha := & \sum_{I \subset [d]} \sum_{\bd{i}, \bd{i}' \in \Jn(I)} \sum_{\bd{j} \in \Jn(I^c)} A_{(\bd{i} \dot\times \bd{j}) \dot+ (\bd{i}' \dot\times \bd{j})} \prod_{l \in I} \left[ X^{(l)}_{\bd{i}_{l}} X^{(l)}_{\bd{i}'_l} - \mathbbm{1}_{\bd{i}_l = \bd{i}'_l} \right] \\
			= &
			\sum_{I \subset [d]} \sum_{\bd{i}, \bd{i}' \in \Jn(I)} \sum_{\bd{j} \in \Jn(I^c)} A_{(\bd{i} \dot\times \bd{j}) \dot+ (\bd{i}' \dot\times \bd{j})} \sum_{J \subset I} \prod_{l \in I \backslash J } (-\mathbbm{1}_{\bd{i}_l = \bd{i}'_l}) \prod_{l \in J} X^{(l)}_{\bd{i}_{l}} X^{(l)}_{\bd{i}'_l}
		\end{align*}
		
		Observing that
		\begin{align*}
			\prod_{l \in I \backslash J } (-\mathbbm{1}_{\bd{i}_l = \bd{i}'_l}) =
			\begin{cases}
				(-1)^{|I \backslash J|} & \text{if } \forall j \in I \backslash J: \bd{i}_l = \bd{i}'_l \\
				0 & \text{otherwise},
			\end{cases}
		\end{align*}
		we can conclude
		\begin{align*}
			\alpha = & \sum_{I \subset [d]} \sum_{J \subset I} \sum_{\substack{\bd{i}, \bd{i}' \in \Jn(J) \\ \bd{k} \in \Jn(I \backslash J)}} \sum_{\bd{j} \in \Jn(I^c)} A_{(\bd{i} \dot\times \bd{j} \dot\times \bd{k}) \dot+ (\bd{i}' \dot\times \bd{j} \dot\times \bd{k} )} (-1)^{|I \backslash J|} \prod_{l \in J} X^{(l)}_{\bd{i}_{l}} X^{(l)}_{\bd{i}'_l} \\
			= &
			\sum_{J \subset [d]} \sum_{I \supset J} (-1)^{|I \backslash J|} \sum_{\bd{i}, \bd{i}' \in \Jn(J)} \sum_{\bd{j} \in \Jn(J^c)} A_{(\bd{i} \dot\times \bd{j}) \dot+ (\bd{i}' \dot\times \bd{j})}  \prod_{l \in J} X^{(l)}_{\bd{i}_{l}} X^{(l)}_{\bd{i}'_l}
		\end{align*}
		
		Lemma \ref{lem:m1_set_powers} yields
		\begin{align*}
			\sum_{I \supset J} (-1)^{|I \backslash J|} = \sum_{I' \subset [d] \backslash J} (-1)^{|I'|} =
			\begin{cases}
				1 & \text{if } J = [d] \\
				0 & \text{otherwise},
			\end{cases}
		\end{align*}
		
		such that
		\begin{align*}
			\alpha = \sum_{\bd{i}, \bd{i}' \in \Jn} A_{\bd{i} \dot+ \bd{i}'}  \prod_{l \in [d]} X^{(l)}_{\bd{i}_{l}} X^{(l)}_{\bd{i}'_l}.
		\end{align*}
		
	\end{proof}

	\begin{lem} \label{lem:remove_mean_inv}
		
		Let $\bd{A} \in \mathbb{R}^{\bd{n}^{\times 2}}$ be an array of order $2d$ and $X^{(1)} \in \mathbb{R}^{n_1}, \dots X^{(d)} \in \mathbb{R}^{n_d}$ independent random vectors with mean $0$, variance $1$ entries. Then for any subset $\emptyset \neq I \subset [d]$, $p \geq 1$,
		
		\begin{align*}
			& \left\|\sum_{\bd{i}, \bd{i}' \in \Jn(I)}
			\sum_{\bd{j} \in \Jn(I^c)}  A_{(\bd{i} \dot\times \bd{j}) \dot\times (\bd{i}' \dot\times \bd{j})} \prod_{l \in I} \left[ X^{(l)}_{\bd{i}_{l}} X^{(l)}_{\bd{i}'_l} - \mathbbm{1}_{\bd{i}_l = \bd{i}'_l} \right] \right\|_{L_p} \\
			\leq & 
			C(|I|) \left\| \sum_{\bd{i}, \bd{i}' \in \Jn} A_{\bd{i} \dot+ \bd{i}'} \prod_{l \in [d]} X^{(l)}_{\bd{i}_l} X^{(l)}_{\bd{i}'_l} - \mathbb{E} \sum_{\bd{i}, \bd{i}' \in \Jn} A_{\bd{i} \dot+ \bd{i}'} \prod_{l \in [d]} X^{(l)}_{\bd{i}_l} X^{(l)}_{\bd{i}'_l} \right\|_{L_p},
			\numberthis \label{eq:rearrange_m1_1}
		\end{align*}
		where $C(|I|)$ is a constant only depending on $|I|$.
	\end{lem}
	
	\begin{proof}
		By the assumptions on the vectors $X^{(l)}$,
		\begin{align*}
			E := \mathbb{E} \sum_{\bd{i}, \bd{i}' \in \Jn} A_{\bd{i} \dot+ \bd{i}'} \prod_{l \in [d]} X^{(l)}_{\bd{i}_l} X^{(l)}_{\bd{i}'_l} = 
			\sum_{\bd{i} \in \Jn} A_{\bd{i} \dot+ \bd{i}}.
		\end{align*}
		
		Since this is exactly the term for $I = \emptyset$ in Lemma \ref{lem:rearrange_m1}, we obtain for the term on the right hand side of (\ref{eq:rearrange_m1_1}),
		\begin{align*}
			b := & \sum_{\bd{i}, \bd{i}' \in \Jn} A_{\bd{i} \dot+ \bd{i}'} \prod_{l \in [d]} X^{(l)}_{\bd{i}_l} X^{(l)}_{\bd{i}'_l} - E
			=
			\sum_{\substack{\emptyset \neq J \subset [d] \\ \bd{i}, \bd{i}' \in \Jn(J) \\ \bd{j} \in \Jn(J^c)}} A_{(\bd{i} \dot\times \bd{j}) \dot+ (\bd{i}' \dot\times \bd{j})} \prod_{l \in J} \left[ X^{(l)}_{\bd{i}_{l}} X^{(l)}_{\bd{i}'_l} - \mathbbm{1}_{\bd{i}_l = \bd{i}'_l} \right] =: \sum_{\substack{J \subset [d] \\ J \neq \emptyset}} S_J.
		\end{align*}
		
		Using these terms, we need to show that $\|S_I\|_{L_p} \leq C(|I|) \|b\|_{L_p}$ for all $\emptyset \neq I \subset [d]$.
		
		Now we prove this by induction over $|I|$. First assume $I = \{l_0\}$. For any $J \neq \emptyset, I$, there exists an $l \in J \backslash I$ and then
		\begin{align*}
			\mathbb{E} \left[ \prod_{l \in J} \left[ X^{(l)}_{\bd{i}_{l}} X^{(l)}_{\bd{i}'_l} - \mathbbm{1}_{\bd{i}_l = \bd{i}'_l} \right] \bigg| X^{(l_0)} \right] = 0
		\end{align*}
		since there is at least one factor whose conditional expectation is $0$.
		
		We conclude
		\begin{align*}
			\mathbb{E} \left| S_I  \right|^p =
			\mathbb{E} \left| S_I +
			\mathbb{E} \left[ \sum_{\substack{J \subset [d]: J \neq \emptyset, I}} S_J  \,\bigg|\, X^{(l_0)} \right] \right|^p =
			\mathbb{E} \left| \mathbb{E} \left[ \sum_{\substack{J \subset [d]: J \neq \emptyset}} S_J  \,\bigg|\, X^{(l_0)} \right]
			\right|^p \leq \mathbb{E} |b|^p,
		\end{align*}
		where we used Jensen's inequality on the conditional expectation in the last step.
		
		Now assume that we have already shown (\ref{eq:rearrange_m1_1}) for all $\emptyset \neq I' \subset [d]$ with $|I'| < |I|$.
		
		For all $J \subset [d]$ such that $J \neq \emptyset, I$, one of the following holds.
		\begin{itemize}
			\item $J \backslash I = \emptyset$, i.e., $J \subset I$: Because $J \neq I$, $|J| < |I|$, so by induction 
			\begin{equation} 
				\|S_J\|_{L_p} \leq C(|J|) \|b\|_{L_p}. \label{eq:rearrange_m1_2}
			\end{equation}
			
			\item $J \backslash I \neq \emptyset$. Since there is an $l' \in J \backslash I$,
			\begin{align}
				\mathbb{E} \left[ \prod_{l \in J} \left[ X^{(l)}_{\bd{i}_{l}} X^{(l)}_{\bd{i}'_l} - \mathbbm{1}_{\bd{i}_l = \bd{i}'_l} \right] \,\bigg|\, (X^{(l)})_{l \in I}  \right] = 0.
				\label{eq:rearrange_m1_3}
			\end{align}
		\end{itemize}
		
		The triangle inequality yields together with (\ref{eq:rearrange_m1_2}), that $\|S_I\|_{L_p} \leq$
		\begin{align*}
			\left\|S_I + \sum_{\substack{J \subset I \\ J \neq \emptyset, I}} S_J\right\|_{L_p} + \sum_{\substack{J \subset I \\ J \neq \emptyset, I}} \|S_J\|_{L_p}
			\leq
			\left\|S_I + \sum_{\substack{J \subset I \\ J \neq \emptyset, I}} S_J\right\|_{L_p} + \left[ \sum_{J \subset I, J \neq \emptyset, I} C(|J|) \right] \|b\|_{L_p}.
		\end{align*}
		
		The first term on the right hand side can be controlled with (\ref{eq:rearrange_m1_3}) and Jensen's inequality,
		\begin{align*}
			\mathbb{E} \left| S_I + \sum_{\substack{J \subset I: J \neq \emptyset, I}} S_J \right|^p = &
			\mathbb{E} \left| S_I + \sum_{\substack{J \subset I: J \neq \emptyset, I}} S_J + \mathbb{E}\left[ \sum_{\substack{J \subset [d]: J \backslash I \neq \emptyset}} S_J \,\Bigg|\, (X^{(l)})_{l \in I} \right] \right|^p \\
			= & 
			\mathbb{E} \left| \mathbb{E}\left[ \sum_{J \subset [d]: J \neq \emptyset} S_J \,\Bigg|\, (X^{(l)})_{l \in I} \right] \right|^p \leq \mathbb{E} |b|^p.
		\end{align*}
		
		So altogether $\|S_I\|_{L_p} \leq C(|I|) \|b\|_{L_p}$ where $C(|I|) := \sum_{J \subset I: J \neq \emptyset, I} C(|J|) + 1$ depends only on $|I|$.
	\end{proof}
	
	Now we introduced all the necessary tools and can prove the lower bound of the main result, Theorem \ref{thm:main}.

	\subsubsection{Proof of Theorem \ref{thm:main}, lower bound} \label{sec:proof_lower_bound}
	
	For any $J \subset I \subset [d]$, define the array $\bd{A}^{(I, J)}$ as in the proof of the upper bound \eqref{eq:definition_AIJ} and
	\begin{align*}
		\alpha^{(I, J)} := & \left\| \sum_{\bd{i}, \bd{i}' \in \Jn(J \cup I^c)} A^{(I, J)}_{\bd{i} \dot+ \bd{i}'} \prod_{l \in I^c \cup J} g^{(l)}_{\bd{i}_l} \bar{g}^{(l)}_{\bd{i}'_l} \right\|_{L_p}
		\numberthis \label{eq:lower_bound_alpha_1}
	\end{align*}

	\textbf{Step 1: Adding off-diagonal terms}
	
	Define independent Rademacher vectors $(\xi^{(l)})_{l \in J}$ which are also independent of the $g^{(1)}, \dots g^{(d)}$, $\bar{g}^{(1)}, \dots, \bar{g}^{(d)}$.
	
	Noting that $\mathbb{E}_\xi[\xi^{(l)}_{\bd{i}_l} \xi^{(l)}_{\bd{i}'_l}] = \mathbbm{1}_{\bd{i}_l = \bd{i}'_l}$, we obtain
	
	\begin{align*}
		& \mathbb{E}_{\xi} \left[ \sum_{\substack{\bd{i}, \bd{i}' \in \Jn(J) \\ \bd{k} \in \Jn(I \backslash J) \\ \bd{j}, \bd{j}' \in \Jn(I^c)}} A_{(\bd{i} \dot\times \bd{j} \dot\times \bd{k}) \dot+ (\bd{i}' \dot\times \bd{j}' \dot\times \bd{k})} \prod_{l \in I^c} g^{(l)}_{\bd{j}_l} \bar{g}^{(l)}_{\bd{j}'_l}  \prod_{l \in J} (\xi^{(l)}_{\bd{i}_l} g^{(l)}_{\bd{i}_l}) (\xi^{(l)}_{\bd{i}'_l} \bar{g}^{(l)}_{\bd{i}'_l}) \right] \\
		& =
		\sum_{\bd{i}, \bd{i}' \in \Jn(J \cup I^c)} A^{(I, J)}_{\bd{i} \dot+ \bd{i}'} \prod_{l \in I^c \cup J} g^{(l)}_{\bd{i}_l} \bar{g}^{(l)}_{\bd{i}'_l}
	\end{align*}
	
	Substituting into (\ref{eq:lower_bound_alpha_1}) and applying Jensen's inequality and Fubini's theorem yields
	\begin{align*}
		(\alpha^{(I, J)})^p = &
		\mathbb{E}_{g, \bar{g}} \left| \mathbb{E}_{\xi} \sum_{\substack{\bd{i}, \bd{i}' \in \Jn(J) \\ \bd{j}, \bd{j}' \in \Jn(I^c)}}
		\sum_{\bd{k} \in \Jn(I \backslash J)} A_{(\bd{i} \dot\times \bd{j} \dot\times \bd{k}) \dot+ (\bd{i}' \dot\times \bd{j}' \dot\times \bd{k})} \prod_{l \in I^c} g^{(l)}_{\bd{j}_l} \bar{g}^{(l)}_{\bd{j}'_l}  \prod_{l \in J} (\xi^{(l)}_{\bd{i}_l} g^{(l)}_{\bd{i}_l}) (\xi^{(l)}_{\bd{i}'_l} \bar{g}^{(l)}_{\bd{i}'_l}) \right|^p \\
		\leq &
		\mathbb{E}_{\xi} \mathbb{E}_{g, \bar{g}} \left| \sum_{\substack{\bd{i}, \bd{i}' \in \Jn(J) \\ \bd{j}, \bd{j}' \in \Jn(I^c)}}
		\sum_{\bd{k} \in \Jn(I \backslash J)} A_{(\bd{i} \dot\times \bd{j} \dot\times \bd{k}) \dot+ (\bd{i}' \dot\times \bd{j}' \dot\times \bd{k})} \prod_{l \in I^c} g^{(l)}_{\bd{j}_l} \bar{g}^{(l)}_{\bd{j}'_l}  \prod_{l \in J} (\xi^{(l)}_{\bd{i}_l} g^{(l)}_{\bd{i}_l}) (\xi^{(l)}_{\bd{i}'_l} \bar{g}^{(l)}_{\bd{i}'_l}) \right|^p
	\end{align*}
	
	By the symmetry of the normal distribution, conditioned on $(\xi^{(l)})_{l \in J}$, $(\xi^{(l)}_{\bd{i}_l} g^{(l)}_{\bd{i}_l}, \xi^{(l)}_{\bd{i}'_l} \bar{g}^{(l)}_{\bd{i}'_l})$ and $(g^{(l)}_{\bd{i}_l}, \bar{g}^{(l)}_{\bd{i}'_l})$ have the same distribution. So we can conclude
	
	\begin{align*}
		\alpha^{(I, J)} \leq 
		\left\| \sum_{\substack{\bd{i}, \bd{i}' \in \Jn(J \cup I^c)}} \sum_{\bd{k} \in \Jn(I \backslash J)} A_{(\bd{i} \dot\times \bd{k}) \dot+ (\bd{i}' \dot\times \bd{k})} \prod_{l \in J \cup I^c} g^{(l)}_{\bd{i}_l} \bar{g}^{(l)}_{\bd{i}'_l} \right\|_{L_p}.
	\end{align*}

	\textbf{Step 2: Inverse Gaussian decoupling}
	
	For every $J \subset I \subset [d]$, we obtain then by the symmetry of $\bd{A}$ and Corollary \ref{cor:gauss_rev_decouple},
	\begin{align*}
		\alpha^{(I, J)} \leq
		\left\| \sum_{\substack{\bd{i}, \bd{i}' \in \Jn(J \cup I^c)}} \sum_{\bd{k} \in \Jn(I \backslash J)} A_{(\bd{i} \dot\times \bd{k}) \dot+ (\bd{i}' \dot\times \bd{k})} \prod_{l \in J \cup I^c} \left[ g^{(l)}_{\bd{i}_l} g^{(l)}_{\bd{i}'_l} - \mathbbm{1}_{\bd{i}_l = \bd{i}'_l} \right] \right\|_{L_p}.
	\end{align*}

	\textbf{Step 3: Removing the mean subtractions in every factor}
	
	Since $I \backslash J \neq [d]$, $J \cup I^c \neq \emptyset$ and Lemma \ref{lem:remove_mean_inv} provides
	\begin{align*}
		\alpha^{(I, J)} \leq 
		C_1(|J \cup I^c|)\left\| \sum_{\bd{i}, \bd{i}' \in \Jn} A_{\bd{i} \dot+ \bd{i}'} \prod_{l \in [d]} g^{(l)}_{\bd{i}_l} g^{(l)}_{\bd{i}'_l} - \mathbb{E} \sum_{\bd{i}, \bd{i}' \in \Jn} A_{\bd{i} \dot+ \bd{i}'} \prod_{l \in [d]} g^{(l)}_{\bd{i}_l} g^{(l)}_{\bd{i}'_l} \right\|_{L_p}.
	\end{align*}
	
	Adding this up over all $J \subset I \subset [d]$, $I \backslash J \neq [d]$ yields
	\begin{align} \label{eq:lower_bound_1}
		\sum_{\substack{J \subset I \subset [d] \\ I \backslash J \neq [d]}} \alpha^{(I, J)} \leq
		C(d) \left\| \sum_{\bd{i}, \bd{i}' \in \Jn} A_{\bd{i} \dot+ \bd{i}'} \prod_{l \in [d]} g^{(l)}_{\bd{i}_l} g^{(l)}_{\bd{i}'_l} - \mathbb{E} \sum_{\bd{i}, \bd{i}' \in \Jn} A_{\bd{i} \dot+ \bd{i}'} \prod_{l \in [d]} g^{(l)}_{\bd{i}_l} g^{(l)}_{\bd{i}'_l} \right\|_{L_p}
	\end{align}
	where $C(d) := \sum_{\substack{J \subset I \subset [d]: I \backslash J \neq [d]}} C_1(|J \cup I^c|)$ depends only on $d$.
	
	\textbf{Step 4: Completing the proof}
	
	Restricting the left hand side in \eqref{eq:lower_bound_1} to the terms in which $J = \emptyset$. The remaining terms $\alpha^{(I, \emptyset)}$ only contain the arrays $\bd{A}^{(I, \emptyset)}$ which are equal to the $\bd{A}^{(I)}$ from the theorem statement. Subsequently, we can bound the $\alpha^{(I, \emptyset)}$ from below using Theorem \ref{thm:gauss_chaos_moments} (similarly to the upper bound) to obtain the lower bound in Theorem \ref{thm:main}.
	
	\subsection{Concentration of $\|A X \|_2$}
	
	In this section, we apply our main results to the concentration of $\|A X\|_2$ where $X = X^{(1)} \otimes \dots \otimes X^{(d)}$ is a Kronecker product of independent vectors with subgaussian entries. The following statement is a direct consequence from Theorem \ref{thm:main} and Lemma \ref{lem:a2b2_ab_comparison}.
	
	\begin{cor} \label{cor:moments_non_squared}
		Let $A \in \mathbb{R}^{n_0 \times N}$ be a matrix where $N = n_1 \dots n_d$ and $X := X^{(1)} \otimes \dots \otimes X^{(d)} \in \R^N$ a random vector as in Theorem \ref{thm:main}.
		
		Let $\bd{B} \in \R^{\bd{n}^{\times 2}}$ be the rearrangement of the matrix $B = A^* A$ as an array with $2 d$ axes. For any $I \subset [d]$, define the array $\bd{B}^{(I)}$ as in \eqref{eq:definition_BI}.
		
		For $T \subset [2 d]$, $1 \leq \kappa \leq 2d$, denote $S(T, \kappa)$ for the set of partitions of $T$ into $\kappa$ sets and $I^c = [d] \backslash I$. Define for any $p \geq 1$ and any $\kappa \in [2d]$,
		\begin{align*}
			m_{p, \kappa} & := \sum_{\substack{I \subset [d] \\ I \neq [d]}}
			\sum_{(I_1, \dots, I_\kappa) \in S((I^c) \cup (I^c + d), \kappa)} \|\bd{B}^{(I)}\|_{I_1, \dots, I_\kappa} \\
			m_p & := L^{2 d} \sum_{\kappa = 1}^{2 d}  \min \left\{ p^\frac{\kappa}{2} \frac{m_{p, \kappa}}{\|A\|_F}, p^{\frac{\kappa}{4}} \sqrt{m_{p, \kappa}}  \right\}
		\end{align*}
		
		Then there is a constant $C(d) > 0$, depending only on $d$, such that for all $p \geq 2$,
		\begin{align*}
			\left\| \|A X\|_2 - \|A\|_F \right\|_{L_p} \leq
			C(d) m_p.
		\end{align*}
		
		If in addition, $X^{(1)} \sim N(0, Id_{n_1}), \dots, X^{(d)} \sim N(0, Id_{n_d})$ are normally distributed (i.e., $L$ is constant) and $\bd{B}$ satisfies the symmetry condition \eqref{eq:symmetry_condition}, then also the lower bound
		\begin{align*}
			\tilde{C}(d) m_p \leq
			\left\| \|A X\|_2 - \|A\|_F \right\|_{L_p}
		\end{align*}
		holds for all $p \geq 2$.  Above, $\tilde{C}(d) > 0$ that depends only on $d$.
	\end{cor}

	\begin{lem} \label{lem:join_sep_norm}
		Let $\bd{A} \in \R^{n_1 \times \dots \times n_d}$. Assume that $I_1, \dots, I_\kappa$ is a partition of $[d]$. Let $\bar{I}_{\kappa} \cup \bar{I}_{\kappa + 1} = I_\kappa$ be a partition into two subsets. Then
		\[
		\|\bd{A}\|_{I_1, \dots, I_{\kappa - 1}, \bar{I}_{\kappa}, \bar{I}_{\kappa+1}} \leq \|\bd{A}\|_{I_1, \dots, I_\kappa} \leq \sqrt{\min\left\{ \prod_{l \in \bar{I}_{\kappa}} n_l, \prod_{l \in \bar{I}_{\kappa + 1}} n_l \right\}} \|\bd{A}\|_{I_1, \dots, I_{\kappa - 1}, \bar{I}_{\kappa}, \bar{I}_{\kappa+1}}.
		\]
	\end{lem}
	
	\begin{proof}
		Take arrays $\bd{\alpha}^{(1)} \in \R^{\bd{n}}(I_1), \dots, \bd{\alpha}^{(\kappa - 1)} \in \R^{\bd{n}}(I_{\kappa - 1}), \bar{\bd{\alpha}}^{(\kappa)} \in \R^{\bd{n}}(\bar{I}_{\kappa}), \bar{\bd{\alpha}}^{(\kappa + 1)} \in \R^{\bd{n}}(\bar{I}_{\kappa + 1})$, with Frobenius norm $1$ each, such that $\|\bd{A}\|_{I_1, \dots, I_{\kappa - 1}, \bar{I}_{\kappa}, \bar{I}_{\kappa + 1}} = 
		\sum_{\bd{i} \in \Jn} A_{\bd{i}} \alpha^{(1)}_{\bd{i}_{I_1}} \dots \alpha^{(\kappa - 1)}_{\bd{i}_{I_{\kappa - 1}}} \bar{\alpha}^{(\kappa)}_{\bd{i}_{\bar{I}_{\kappa}}} \bar{\alpha}^{(\kappa + 1)}_{\bd{i}_{\bar{I}_{\kappa + 1}}}$. Now define $\bd{\alpha}^{(\kappa)} \in \R^{\bd{n}}(I_\kappa)$ by $\alpha^{(\kappa)}_{\bd{i}} = \bar{\alpha}^{(\kappa)}_{\bd{i}_{\bar{I}_{\kappa}}} \bar{\alpha}^{(\kappa + 1)}_{\bd{i}_{\bar{I}_{\kappa + 1}}}$ for every $\bd{i} \in \Jn(I_\kappa)$. Then $\|\bd{\alpha}^{(\kappa)}\|_{2} = 1$ and by the definition of $\|\cdot\|_{I_1, \dots, I_\kappa}$ as the supremum over $\bd{\alpha}^{(1)}, \dots, \bd{\alpha}^{(\kappa)}$, we obtain
		\begin{align*}
			\|\bd{A}\|_{I_1, \dots, I_{\kappa - 1}, \bar{I}_{\kappa}, \bar{I}_{\kappa + 1}} = \sum_{\bd{i} \in \Jn} A_{\bd{i}} \alpha^{(1)}_{\bd{i}_{I_1}} \dots \alpha^{(\kappa)}_{\bd{i}_{I_\kappa}}
			\leq \|\bd{A}\|_{I_1, \dots, I_{\kappa}},
		\end{align*}
		which proves the first inequality.
		
		To prove the second inequality, take arrays $\bd{\alpha}^{(1)} \in \R^{\bd{n}}(I_1), \dots, \bd{\alpha}^{(\kappa)} \in \R^{\bd{n}}(I_\kappa)$ such that
		\begin{align*}
			\|\bd{A}\|_{I_1, \dots, I_\kappa} = \sum_{\bd{i} \in \Jn} A_{\bd{i}} \alpha^{(1)}_{\bd{i}_{I_1}} \dots \alpha^{(\kappa)}_{\bd{i}_{I_\kappa}}.
		\end{align*}
		Now define $\tilde{\bd{A}} \in \R^{\bd{n}}(I_\kappa)$ such that for all $\bd{i} \in \Jn(\bar{I}_\kappa), \bd{j} \in \Jn(\bar{I}_{\kappa + 1})$,
		\[
		\tilde{\bd{A}}_{\bd{i} \dot\times \bd{j}} = \sum_{\bd{k} \in \Jn([d] \backslash I_\kappa)} A_{\bd{i} \dot\times \bd{j} \dot\times \bd{k}} \alpha^{(1)}_{\bd{k}_{I_1}} \dots \alpha^{(\kappa - 1)}_{\bd{k}_{I_{\kappa - 1}}}.
		\]
		
		For $N_1 := \prod_{l \in \bar{I}_{\kappa}} n_l$ and $N_2 := \prod_{l \in \bar{I}_{\kappa + 1}} n_l$, we can interpret $\tilde{\bd{A}}$ as a matrix $\tilde{A} \in \R^{N_1 \times N_2}$ with rows indexed by $\bd{i} \in \Jn(\bar{I}_\kappa)$ and columns indexed by $\bd{j} \in \Jn(\bar{I}_{\kappa + 1})$.
		
		Then
		\begin{align*}
			\|\tilde{A}\|_F & = \sup_{\bd{\beta} \in \R^{\bd{n}}(I_\kappa), \|\bd{\beta}\|_2 = 1} \sum_{\bd{i} \in \Jn(I_\kappa)} \tilde{A}_{\bd{i}} \beta_{\bd{i}}, \\
			\|\tilde{A}\|_{2 \rightarrow 2} & = \sup_{\substack{\bd{\beta}^{(1)} \in \R^{\bd{n}}(\bar{I}_\kappa), \bd{\beta}^{(2)} \in \R^{\bd{n}}(\bar{I}_{\kappa + 1}), \\ \|\bd{\beta}^{(1)}\|_2 = \|\bd{\beta}^{(2)}\|_2 = 1}} \sum_{\substack{\bd{i} \in \Jn(\bar{I}_\kappa) \\ \bd{j} \in \Jn(\bar{I}_{\kappa + 1}) }} \tilde{A}_{\bd{i} \dot\times \bd{j}} \beta_{\bd{i}}^{(1)} \beta_{\bd{j}}^{(2)},
		\end{align*}
		such that
		\begin{align*}
			\|\tilde{A}\|_F & = \sup_{\bd{\beta} \in \R^{\bd{n}}(I_\kappa), \|\bd{\beta}\|_2 = 1} \sum_{\bd{i} \in \Jn(I_\kappa)} \sum_{\bd{k} \in \Jn([d] \backslash I_\kappa)} A_{\bd{i} \dot\times \bd{k}} \alpha^{(1)}_{\bd{k}_{I_1}} \dots \alpha^{(\kappa - 1)}_{\bd{k}_{I_{\kappa - 1}}} \beta_{\bd{i}} \\
			& = \sup_{\bd{\beta} \in \R^{\bd{n}}(I_\kappa), \|\bd{\beta}\|_2 = 1} \sum_{\bd{i} \in \Jn} A_{\bd{i}} \alpha^{(1)}_{\bd{i}_{I_1}} \dots \alpha^{(\kappa - 1)}_{\bd{i}_{I_{\kappa - 1}}} \beta_{\bd{i}_{I_\kappa}},
		\end{align*}
		where by definition the maximum is attained at $\bd{\beta} = \bd{\alpha}^{(\kappa)}$, implying
		\begin{align}
			\|\tilde{A}\|_F = \|\bd{A}\|_{I_1, \dots, I_\kappa}. \label{eq:norm_ineq_fro}
		\end{align}
		For the spectral norm, we obtain from the definition of $\|\cdot\|_{I_1, \dots, I_{\kappa - 1}, \bar{I}_{\kappa}, \bar{I}_{\kappa + 1}}$,
		\begin{align*}
			\|\tilde{A}\|_{2 \rightarrow 2} & = \sup_{\substack{\bd{\beta}^{(1)} \in \R^{\bd{n}}(\bar{I}_\kappa), \bd{\beta}^{(2)} \in \R^{\bd{n}}(\bar{I}_{\kappa + 1}), \\ \|\bd{\beta}^{(1)}\|_2 = \|\bd{\beta}^{(2)}\|_2 = 1}} \sum_{\substack{\bd{i} \in \Jn(\bar{I}_\kappa) \\ \bd{j} \in \Jn(\bar{I}_{\kappa + 1}) }} \sum_{\bd{k} \in \Jn([d] \backslash I_\kappa)} A_{\bd{i} \dot\times \bd{j} \dot\times \bd{k}} \alpha^{(1)}_{\bd{k}_{I_1}} \dots \alpha^{(\kappa - 1)}_{\bd{k}_{I_{\kappa - 1}}} \beta_{\bd{i}}^{(1)} \beta_{\bd{j}}^{(2)} \\
			& = \sup_{\substack{\bd{\beta}^{(1)} \in \R^{\bd{n}}(\bar{I}_\kappa), \bd{\beta}^{(2)} \in \R^{\bd{n}}(\bar{I}_{\kappa + 1}), \\ \|\bd{\beta}^{(1)}\|_2 = \|\bd{\beta}^{(2)}\|_2 = 1}} \sum_{\bd{i} \in \Jn} A_{\bd{i}} \alpha^{(1)}_{\bd{i}} \dots \alpha^{(\kappa - 1)}_{\bd{i}_{I_{\kappa - 1}}} \beta_{\bd{i}_{\bar{I}_{\kappa}}}^{(1)} \beta_{\bd{j}_{\bar{i}_{\kappa + 1}}}^{(2)} \\
			& \leq \|\bd{A}\|_{I_1, \dots, I_{\kappa - 1}, \bar{I}_{\kappa}, \bar{I}_{\kappa + 1}}.
			\numberthis \label{eq:norm_ineq_spec}
		\end{align*}
		The second inequality now follows from \eqref{eq:norm_ineq_fro}, \eqref{eq:norm_ineq_spec} and the general property of matrices that
		\begin{align*}
			\|\tilde{A}\|_F \leq \sqrt{\mathrm{rank}(\tilde{A})} \|\tilde{A}\|_{2 \rightarrow 2}
			\leq \sqrt{\min\{N_1, N_2\}} \|\tilde{A}\|_{2 \rightarrow 2}.
		\end{align*}
		
	\end{proof}
	
	\begin{lem} \label{lem:norm_B_I}
		Let $\bd{B} \in \R^{\bd{n}^{\times 2}}$, $I \subset [d]$. Define $\bd{B}^{(I)}$ as in \eqref{eq:definition_BI}.
		
		Let $I_1, \dots, I_{\kappa}$ be a partition of $([d] \backslash I) \cup (d + ([d] \backslash I))$. Let $I_{\kappa+1}, \dots, I_{\kappa + |I|}$ be the sets $\{j, j+d\}$ for every $j \in I$. Then $I_1, \dots, I_{\kappa + |I|}$ is a paritition of $[2d]$ and
		\[
		\|\bd{B}^{(I)}\|_{I_1, \dots, I_{\kappa}} \leq \sqrt{\prod_{l \in I} n_l } \|\bd{B}\|_{I_1, \dots, I_{\kappa + |I|}}
		\]
	\end{lem}
	
	\begin{proof}
		Take $\bd{\alpha}^{(1)} \in \R^{\bd{n}^{\times 2}}(I_1), \dots, \bd{\alpha}^{(\kappa)} \in \R^{\bd{n}^{\times 2}}(I_\kappa)$, all having a Frobenius norm of $1$, such that
		
		\begin{align*}
			& \|\bd{B}^{(I)}\|_{I_1, \dots, I_\kappa}
			= \sum_{\bd{i} \in \bd{J}^{\bd{n}^{\times 2}}(I^c \cup (I^c + d))} B^{(I)}_{\bd{i}} \alpha^{(1)}_{\bd{i}_{I_1}} \dots \alpha^{(\kappa)}_{\bd{i}_{I_\kappa}} \\
			& = \sum_{\bd{i} \in \bd{J}^{\bd{n}^{\times 2}}(I^c \cup (I^c + d))} \sum_{\bd{k} \in \Jn(I)} B_{\bd{i} \dot\times (\bd{k} \dot+ \bd{k})} \alpha^{(1)}_{\bd{i}_{I_1}} \dots \alpha^{(\kappa)}_{\bd{i}_{I_\kappa}}
			= \sum_{\bd{i} \in \bd{J}^{\bd{n}^{\times 2}}} B_{\bd{i}} \alpha^{(1)}_{\bd{i}_{I_1}} \dots \alpha^{(\kappa)}_{\bd{i}_{I_\kappa}} \mathbbm{1}_{\forall l \in I: \bd{i}_l = \bd{i}_{l + d}}.
			\numberthis \label{eq:norm_B_I_1}
		\end{align*}
		
		Now define $\bd{\alpha}^{(\kappa + 1)} \in \R^{\bd{n}^{\times 2}}(\{j_1, j_1 + d\}), \dots, \bd{\alpha}^{(\kappa + |I|)} \in \R^{\bd{n}^{\times 2}}(\{j_{|I|}, j_{|I|} + d\})$ (where $I = \{j_1, \dots, j_{|I|}\}$) such that for all $r \in [|I|]$ and $\bd{i} \in \bd{J}^{\bd{n}^{\times 2}}(\{j_r, j_r + d\})$,
		\begin{align*}
			\alpha^{(\kappa + r)}_{\bd{i}} =
			\begin{cases}
				\frac{1}{\sqrt{n_{j_r}}} & \text{if } \bd{i}_{j_r} = \bd{i}_{j_r + d} \\
				0 & \text{otherwise}.
			\end{cases}
		\end{align*}
		
		Then for  $\bd{i} \in \bd{J}^{\bd{n}^{\times 2}}(I \cup (I + d))$
		\begin{align*}
			\alpha^{(\kappa + 1)}_{\bd{i}_{I_{\kappa + 1}}} \dots \alpha^{(\kappa + |I|)}_{\bd{i}_{I_{\kappa + |I|}}} = \frac{1}{\sqrt{\prod_{l \in I} n_l}} \mathbbm{1}_{\forall l \in I: \bd{i}_l = \bd{i}_{l + d}}
		\end{align*}
		
		Substituting this into \eqref{eq:norm_B_I_1} yields
		\begin{align*}
			\|\bd{B}^{(I)}\|_{I_1, \dots, I_\kappa}
			& = \sqrt{\prod_{l \in I} n_l} \sum_{\bd{i} \in \bd{J}^{\bd{n}^{\times 2}}} B_{\bd{i}} \alpha^{(1)}_{\bd{i}_{I_1}} \dots \alpha^{(\kappa)}_{\bd{i}_{I_\kappa}} \alpha^{(\kappa + 1)}_{\bd{i}_{I_{\kappa + 1}}} \dots \alpha^{(\kappa + |I|)}_{\bd{i}_{I_{\kappa + |I|}}}
			\leq \sqrt{\prod_{l \in I} n_l} \|\bd{B}\|_{I_1, \dots, I_{\kappa + |I|}}
		\end{align*}
	\end{proof}
	
	Using the aforementioned results, we can give the proof of Theorem \ref{thm:Ax_concentration} about $\|A (X^{(1)} \otimes \dots \otimes X^{(d)})\|_2$ in which we find suitable bounds for all the tensor norms of $A^* A$ in terms of $\|A\|_{2 \rightarrow 2}$ and $\|A\|_F$.
	
	\begin{proof}[Proof of Theorem \ref{thm:Ax_concentration}]
		Let $B := A^* A \in \R^{n^d \times n^d}$ and $\bd{B} \in \R^{\bd{n}^{\times 2}}$ be the corresponding array of order $2 d$ obtained by rearranging $B$ for $\bd{n} = (n, \dots, n)$. Note that here the dimensions along all axes are equal. For $I \subset [2 d]$, define $\bd{B}^{(I)}$ as in Corollary \ref{cor:moments_non_squared}.
		
		\textbf{Step 1: Showing the norm inequalities}
		\begin{equation}
			\|\bd{B}^{(I)}\|_{I_1, \dots, I_\kappa} \leq n^{\frac{|I|}{2}} \|B\|_F
			\qquad
			\|\bd{B}^{(I)}\|_{I_1, \dots, I_\kappa} \leq n^{d - \frac{\kappa}{2}} \|B\|_{2 \rightarrow 2}.
			\label{eq:B_I_norm_inequalities}
		\end{equation}
		
		In both cases, we start by extending $I_1, \dots, I_\kappa$ to $I_1, \dots, I_{\kappa + |I|}$ as in Lemma \ref{lem:norm_B_I}, obtaining
		\begin{equation}
			\|\bd{B}^{(I)}\|_{I_1, \dots, I_\kappa} \leq n^\frac{|I|}{2} \|\bd{B}\|_{I_1, \dots, I_{\kappa + |I|}}
			\label{eq:norm_B_I_2}
		\end{equation}
		Then the first inequality of (\ref{eq:B_I_norm_inequalities}) follows by repeatedly joining all the sets $I_1, \dots, I_{\kappa + |I|}$ in the sense of Lemma \ref{lem:join_sep_norm} (first inequality) yielding $\|\bd{B}\|_{I_1, \dots, I_{\kappa + |I|}} \leq \|\bd{B}\|_{[2 d]}  = \|B\|_F$.
		
		For the second inequality in (\ref{eq:B_I_norm_inequalities}), we distinguish two cases. First assume that $\kappa \leq d - |I|$. Then $|I| \leq d - \kappa$. Since $B$ is a matrix in $\R^{n^d \times n^d}$, $\|B\|_{2 \rightarrow 2} \leq n^\frac{d}{2} \|B\|_F$ and with the first inequality in (\ref{eq:B_I_norm_inequalities}), we obtain
		\[
		\|\bd{B}^{(I)}\|_{I_1, \dots, I_\kappa} \leq n^{\frac{|I|}{2}} n^{\frac{d}{2}} \|B\|_{2 \rightarrow 2} \leq
		n^{\frac{d - \kappa}{2}} n^{\frac{d}{2}} \|B\|_{2 \rightarrow 2} = n^{d - \frac{\kappa}{2}} \|B\|_{2 \rightarrow 2}.
		\]
		
		In the other case that $\kappa > d - |I|$, denote $\kappa'$ for the number of sets among $I_1, \dots, I_\kappa$ that only contain one element. Since each of the other sets must contain at least two elements, this leads to the inequality
		\begin{align*}
			\kappa' + 2 (\kappa - \kappa') & \leq  |I_1 \cup \dots \cup I_\kappa| \qquad
			\Rightarrow  2 \kappa - \kappa' & \leq  2(d - |I|) \qquad
			\Rightarrow  \kappa' & \geq  2(\kappa - d + |I|).
		\end{align*}
		
		This implies that among $I_1, \dots, I_\kappa$, there must be at least $\kappa - d + |I|$ sets with exactly one element that are all contained in $[d]$ or all contained in $[2 d] \backslash [d]$. Without loss of generality, we can assume that these are $I_1, \dots, I_{\kappa - d + |I|}$. Now take the unions $\bar{I}_1 := I_1 \cup \dots \cup I_{\kappa - d + |I|}$ and $\bar{I}_2 := I_{\kappa - d + |I| + 1} \cup \dots \cup I_{\kappa + |I|}$. With (\ref{eq:norm_B_I_2}) and the first inequality of Lemma \ref{lem:join_sep_norm}, we obtain
		\[
		\|\bd{B}^{(I)}\|_{I_1, \dots, I_\kappa} \leq n^\frac{|I|}{2} \|\bd{B}\|_{\bar{I}_1, \bar{I}_2}.
		\]
		
		Now split up $\bar{I}_2$ into $\bar{I}_{2, 1} := \bar{I}_2 \cap [d]$ and $\bar{I}_{2, 2} := \bar{I}_2 \cap ([2 d] \backslash [d])$. If neither $\bar{I}_{2, 1}$ nor $\bar{I}_{2, 2}$ is empty, then with the second inequality of Lemma \ref{lem:join_sep_norm}, we obtain
		\[
		\|\bd{B}^{(I)}\|_{I_1, \dots, I_\kappa} \leq n^\frac{|I|}{2} n^{\frac{1}{2} \min\{|\bar{I}_{2, 1}|, |\bar{I}_{2, 2}|\}} \|\bd{B}\|_{\bar{I}_1, \bar{I}_{2, 1}, \bar{I}_{2, 2}} \leq
		n^{\frac{|I|}{2} + \frac{1}{2} \min\{|\bar{I}_{2, 1}|, |\bar{I}_{2, 2}|\} } \|\bd{B}\|_{[d], ([2 d] \backslash [d])},
		\]
		where in the last step we used the first inequality in Lemma \ref{lem:join_sep_norm} with the fact that $\bar{I}_1 \cup \bar{I}_{2, 1} \cup \bar{I}_{2, 2} = [2 d]$ and each of these three sets is contained in either $[d]$ or $[2 d] \backslash [d]$. Note that the inequality between the first and the third term still holds in the case that $\bar{I}_{2, 1}$ or $\bar{I}_{2, 2}$ is empty and thus Lemma 8.4 cannot be applied in the first step.
		
		Now assume $\bar{I}_1 \subset [d]$ (otherwise $\bar{I}_1 \subset [2 d] \backslash [d]$ and the proof works analogously). Then $\bar{I}_1 \cup \bar{I}_{2, 1} = [d]$ and $\bar{I}_{2, 1} = [2 d] \backslash [d]$. So $\min\{|\bar{I}_{2, 1}|, |\bar{I}_{2, 2}|\} = |\bar{I}_{2, 1}| = d - |\bar{I}_1| = d - (\kappa - d + |I|) = 2d - \kappa - |I|$. This implies
		\[
		\|\bd{B}^{(I)}\|_{I_1, \dots, I_\kappa} \leq n^{\frac{|I|}{2} + \frac{1}{2}(2d - \kappa - |I|) } \|\bd{B}\|_{[d], ([2 d] \backslash [d])} = n^{d - \frac{\kappa}{2}} \|B\|_{2 \rightarrow 2}.
		\]
		This completes the proof of (\ref{eq:B_I_norm_inequalities}).
		
		\textbf{Step 2: Moment and tail bounds}
		
		Now, use Corollary \ref{cor:moments_non_squared} and its notation of $m_{p, \kappa}$ and $m_p$. The number of terms in the sum of the definition of $m_{p, \kappa}$ only depends on $d$. This fact together with (\ref{eq:B_I_norm_inequalities}) leads to
		\begin{align*}
			m_{p, \kappa} \leq & C_1(d) \max_{I \subset [d], I \neq [d]} n^\frac{|I|}{2} \|B\|_F = C_1(d) n^\frac{d - 1}{2} \|B\|_F \leq C_1(d) n^\frac{d - 1}{2} \|A\|_{2} \|A\|_F. \\
			m_{p, \kappa} \leq & C_1(d) n^{d - \frac{\kappa}{2}} \|B\|_{2 \rightarrow 2} = C_1(d) n^{d - \frac{\kappa}{2}} \|A\|_{2 \rightarrow 2}^2,
		\end{align*}
		where $C_1(d)$ is a constant depending only on $d$. Furthermore, we obtain
		\begin{align*}
			& m_p \leq C_1(d) L^{2 d} \\
			& \cdot \sum_{\kappa = 1}^{2 d} \min \left\{p^\frac{\kappa}{2} n^\frac{d - 1}{2} \|A\|_{2 \rightarrow 2}, p^\frac{\kappa}{2} n^{d - \frac{\kappa}{2}} \frac{\|A\|_{2 \rightarrow 2}^2}{\|A\|_F}, p^\frac{\kappa}{4} n^{\frac{d - 1}{4}} \sqrt{\|A\|_{2 \rightarrow 2} \|A\|_F}, p^\frac{\kappa}{4} n^{\frac{d}{2} - \frac{\kappa}{4}} \|A\|_{2 \rightarrow 2}   \right\}.
		\end{align*}
		
		Since this is an upper bound on the $L_p$ norm of $\|A X\|_2 - \|A\|_F$, Lemma \ref{lem:moment_tail_bound} implies
		\begin{align*}
			\mathbb{P}\left( \left| \|A X\|_2 - \|A\|_F \right| > t \right) \leq e^2 \exp \left( - C_2(d) \min_{\kappa \in [2 d]} \beta_\kappa \right)
		\end{align*}
		where
		\begin{align*}
			\beta_\kappa := \max\Biggl\{ & \left( \frac{t}{n^\frac{d - 1}{2} \|A\|_{2 \rightarrow 2} } \right)^\frac{2}{\kappa}, \left( \frac{t \|A\|_F}{n^{d - \frac{\kappa}{2}}\|A\|_{2 \rightarrow 2}^2} \right)^{\frac{2}{\kappa}}, \\
			& \left( \frac{t}{n^{\frac{d - 1}{4}} \sqrt{\|A\|_{2 \rightarrow 2} \|A\|_F}} \right)^\frac{4}{\kappa}, \left( \frac{t}{n^{\frac{d}{2} - \frac{\kappa}{4}} \|A\|_{2 \rightarrow 2} } \right)^\frac{4}{\kappa} \Biggr\}. \numberthis \label{eq:bk_moment_bound}
		\end{align*}
		
		Now, for each of multiple different ranges of $t$, we select one of the four terms in \eqref{eq:bk_moment_bound}.
		
		\textbf{Step 3: Bound for $t \leq n^\frac{d}{2} \|A\|_{2 \rightarrow 2}$}
		
		For $\kappa = 1$, we obtain using the first term in \eqref{eq:bk_moment_bound}, $\beta_1 \geq \left( {t} / {(n^{\frac{d - 1}{2}} \|A\|_{2 \rightarrow 2})} \right)^2$.
		
		For $\kappa \geq 2$, we can use the fourth term in \eqref{eq:bk_moment_bound} to show the same bound because
		\begin{align*}
			\beta_\kappa \geq \left( \frac{t}{n^{\frac{d}{2} - \frac{\kappa}{4}} \|A\|_{2 \rightarrow 2} } \right)^\frac{4}{\kappa} = n \left( \frac{t}{n^{\frac{d}{2}} \|A\|_{2 \rightarrow 2} } \right)^\frac{4}{\kappa} \geq n \left( \frac{t}{n^{\frac{d}{2}} \|A\|_{2 \rightarrow 2} } \right)^2 = \frac{t^2}{n^{d - 1} \|A\|_{2 \rightarrow 2}^2}.
		\end{align*}
		
		This implies that
		\begin{align*}
			\mathbb{P}\left( \left| \|A X\|_2 - \|A\|_F \right| > t \right) \leq e^2 \exp \left( - C_2(d) \frac{t^2}{n^{d - 1} \|A\|_{2 \rightarrow 2}^2 } \right).
		\end{align*}
		
		\textbf{Step 5: Bound for $t \geq  n^\frac{d}{2} \|A\|_{2 \rightarrow 2}$}
		
		For all $\kappa \in [2 d]$, using the fourth term in \eqref{eq:bk_moment_bound} yields
		\begin{align*}
			\beta_\kappa \geq \left( \frac{t}{n^{\frac{d}{2} - \frac{\kappa}{4}} \|A\|_{2 \rightarrow 2} } \right)^\frac{4}{\kappa} = n \left( \frac{t}{n^{\frac{d}{2}} \|A\|_{2 \rightarrow 2} } \right)^\frac{4}{\kappa} \geq n \left( \frac{t}{n^{\frac{d}{2}} \|A\|_{2 \rightarrow 2} } \right)^\frac{4}{2 d} = \left( \frac{t}{\|A\|_{2 \rightarrow 2} } \right)^\frac{2}{d},
		\end{align*}
		such that
		\begin{align*}
			\mathbb{P}\left( \left| \|A X\|_2 - \|A\|_F \right| > t \right) \leq e^2 \exp \left( - C_2(d) \left( \frac{t}{\|A\|_{2 \rightarrow 2} } \right)^\frac{2}{d} \right).
		\end{align*}
		
		\textbf{Step 6: Bound for $n^{\frac{d - 1}{4}} \|A\|_{2 \rightarrow 2} \leq t \leq n^{\frac{d - 1}{4}} \|A\|_F$ }
		
		Using the third term in \eqref{eq:bk_moment_bound}, we obtain that
		\begin{align*}
			\beta_\kappa & \geq  \left( \frac{t^2}{n^{\frac{d - 1}{2}} \|A\|_{2 \rightarrow 2} \|A\|_F} \right)^\frac{2}{\kappa} \geq \left( \frac{t n^{\frac{d - 1}{4}} \|A\|_{2 \rightarrow 2}} {n^{\frac{d - 1}{2}} \|A\|_{2 \rightarrow 2} \|A\|_F} \right)^\frac{2}{\kappa} =
			\left( \frac{t} {n^{\frac{d - 1}{4}} \|A\|_F} \right)^\frac{2}{\kappa} \geq
			\frac{t^2} {n^{\frac{d - 1}{2}} \|A\|_F^2},
		\end{align*}
		implying
		\begin{align*}
			\mathbb{P}\left( \left| \|A X\|_2 - \|A\|_F \right| > t \right) \leq e^2 \exp \left( - C_2(d) \frac{t^2} {n^{\frac{d - 1}{2}} \|A\|_F^2} \right).
		\end{align*}		
	\end{proof}
	
	\section{Discussion}
	
	Our main result Theorem \ref{thm:Ax_concentration} controls $\|A X \|_2$ and thus extends a recent concentration result for random tensors by Vershynin \cite{vershynin_random_tensors} to hold for all $t \geq 0$.  Even within the range of $t$ already covered by \cite{vershynin_random_tensors}, our result can provide stronger bounds, in particular for matrices whose $\|\cdot\|_F$ is smaller in relation to their $\|\cdot\|_{2 \rightarrow 2}$ norm such as low-rank matrices. In addition, with Corollary \ref{cor:moments_non_squared}, we provide moment bounds for this situation which are provably tight up to constant factors depending on $d$.
	
	As mentioned in the introduction, Theorem \ref{thm:main} could also be derived from Theorem 1.4 in \cite{adamczak2015concentration} which gives moment bounds in terms of all expected partial derivatives of the chaos with respect to all entries of $X^{(1)}, \dots, X^{(d)}$. Due to the large number of these derivatives, we believe that the proof presented in this work should be more insightful and directly usable. Our approach also provides the decoupling statement of Theorem \ref{thm:decoupling_total} which has been applied in \cite{kronecker_jl_preprint} and generalizes the result in \cite{riptojl} on constructing Johnson-Lindenstrauss embeddings from matrices satisfying the restricted isometry property to Johnson-Lindenstrauss embeddings with a fast transformation of Kronecker products.  The work in \cite{kronecker_jl_preprint}  is provably optimal and is made possible by the results of this work.


	\section*{Acknowledgments}
		R.W. is supported by AFOSR MURI FA9550-19-1-0005, NSF DMS 1952735, and NSF IFML 2019844.
		S.B. and F.K. have been supported by the German Science Foundation (DFG) in the context of the Emmy-Noether Junior Research Group KR4512/1-2.
		R.W. and F.K. gratefully acknowledge support from the Institute for Advanced Study, where this project was initiated.


	
	\printbibliography

@article{kronecker_jl,
	author = {Jin, Ruhui and Kolda, Tamara G and Ward, Rachel},
	title = "{Faster Johnson–Lindenstrauss transforms via Kronecker products}",
	journal = {Information and Inference: A Journal of the IMA},
	year = {2020},
	month = {10},
	issn = {2049-8772},
	note = {iaaa028}
}

@article{battaglino2018practical,
  title={A practical randomized CP tensor decomposition},
  author={Battaglino, Casey and Ballard, Grey and Kolda, Tamara G},
  journal={SIAM Journal on Matrix Analysis and Applications},
  volume={39},
  number={2},
  pages={876--901},
  year={2018},
  publisher={SIAM}
}

@article{latala_gauss_chaos,
 ISSN = {00911798},
 URL = {http://www.jstor.org/stable/25449955},
 author = {Rafał Latała},
 journal = {The Annals of Probability},
 number = {6},
 pages = {2315--2331},
 publisher = {Institute of Mathematical Statistics},
 title = {Estimates of Moments and Tails of Gaussian Chaoses},
 volume = {34},
 year = {2006}
}

@article{riptojl,
 author = {F. Krahmer and R. Ward},
 title = {New and Improved Johnson-Lindenstrauss Embeddings via the Restricted Isometry Property},
 journal = {SIAM Journal on Mathematical Analysis},
 volume = {43},
 number = {3},
 year = {2011},
 pages = {1269--1281}
}

@inbook{vershynin_hdp, 
	place={Cambridge}, 
	series={Cambridge Series in Statistical and Probabilistic Mathematics}, 
	title={Quadratic Forms, Symmetrization, and Contraction}, 
	booktitle={High-Di\-men\-sion\-al Probability: An Introduction with Applications in Data Science}, 
	publisher={Cambridge University Press}, 
	author={Vershynin, Roman},
	year={2018}, pages={127–146},
	collection={Cambridge Series in Statistical and Probabilistic Mathematics}
}

@book{random_series,
	author = {S. Kwapien and W. Woyczynski},
	title = {Random Series and Stochastic Integrals: Single and Multiple},
	publisher = {Birkhäuser Boston},
	year = {1992}
}

@article{hanson1971,
author = "Hanson, D. L. and Wright, F. T.",
doi = "10.1214/aoms/1177693335",
fjournal = "Annals of Mathematical Statistics",
journal = "Ann. Math. Statist.",
month = "06",
number = "3",
pages = "1079--1083",
publisher = "The Institute of Mathematical Statistics",
title = "A Bound on Tail Probabilities for Quadratic Forms in Independent Random Variables",
url = "https://doi.org/10.1214/aoms/1177693335",
volume = "42",
year = "1971"
}

@article{rudelson2013,
author = "Rudelson, Mark and Vershynin, Roman",
doi = "10.1214/ECP.v18-2865",
fjournal = "Electronic Communications in Probability",
journal = "Electron. Commun. Probab.",
pages = "9 pp.",
pno = "82",
publisher = "The Institute of Mathematical Statistics and the Bernoulli Society",
title = "Hanson-Wright inequality and sub-gaussian concentration",
url = "https://doi.org/10.1214/ECP.v18-2865",
volume = "18",
year = "2013"
}

@book{vershynin_2018, place={Cambridge}, series={Cambridge Series in Statistical and Probabilistic Mathematics}, title={High-Dimensional Probability: An Introduction with Applications in Data Science}, DOI={10.1017/9781108231596}, publisher={Cambridge University Press}, author={Vershynin, Roman}, year={2018}, collection={Cambridge Series in Statistical and Probabilistic Mathematics}}

@book{foucart_rauhut,
author = {Foucart, Simon and Rauhut, Holger},
title = {A Mathematical Introduction to Compressive Sensing},
year = {2013},
isbn = {0817649476},
publisher = {Birkh\"{a}user Basel}
}

@article{gauss_decoupling_2,
author = "Miguel A. Arcones and Evarist Giné",
journal = "Journal of Theoretical Probability",
pages = "101-122",
title = "On decoupling, series expansions, and tail behavior of chaos processes",
volume = "6",
year = "1993"
}

@article{gauss_decoupling_1,
author = "Kwapien, Stanislaw",
doi = "10.1214/aop/1176992081",
fjournal = "Annals of Probability",
journal = "Ann. Probab.",
month = "07",
number = "3",
pages = "1062--1071",
publisher = "The Institute of Mathematical Statistics",
title = "Decoupling Inequalities for Polynomial Chaos",
url = "https://doi.org/10.1214/aop/1176992081",
volume = "15",
year = "1987"
}

@inproceedings{oblivious_sketching,
  author={Thomas D. Ahle and Michael Kapralov and Jakob Bæk Tejs Knudsen and Rasmus Pagh and Ameya Velingker and David P. Wood\-ruff and Amir Zandieh},
  title={Oblivious Sketching of High-Degree Polynomial Kernels},
  year={2020},
  cdate={1577836800000},
  pages={141-160},
  url={https://doi.org/10.1137/1.9781611975994.9},
  booktitle={SODA}
}

@article{decoupling_convex,
author = {Victor H. de la Pena},
title = {{Decoupling and Khintchine's Inequalities for $U$-Statistics}},
volume = {20},
journal = {The Annals of Probability},
number = {4},
publisher = {Institute of Mathematical Statistics},
pages = {1877 -- 1892},
year = {1992}
}

@article{rademacher_higher_order,
author = {Stéphane Boucheron and Olivier Bousquet and Gábor Lugosi and Pascal Massart},
title = {{Moment inequalities for functions of independent random variables}},
volume = {33},
journal = {The Annals of Probability},
number = {2},
publisher = {Institute of Mathematical Statistics},
pages = {514 -- 560},
year = {2005}
}

@article{chaos_log_concave,
     author = {Adamczak, Rados\l aw and Lata\l a, Rafa\l },
     title = {Tail and moment estimates for chaoses generated by symmetric random variables with logarithmically concave tails},
     journal = {Annales de l'I.H.P. Probabilit\'es et statistiques},
     pages = {1103--1136},
     publisher = {Gauthier-Villars},
     volume = {48},
     number = {4},
     year = {2012}
}

@article{vershynin_random_tensors,
author = {Roman Vershynin},
title = {{Concentration inequalities for random tensors}},
volume = {26},
journal = {Bernoulli},
number = {4},
publisher = {Bernoulli Society for Mathematical Statistics and Probability},
pages = {3139 -- 3162},
year = {2020}
}

@article{McConnel_Taqqu_Decoupling,
	author = {Terry R. McConnell and Murad S. Taqqu},
	title = {{Decoupling Inequalities for Multilinear Forms in Independent Symmetric Random Variables}},
	volume = {14},
	journal = {The Annals of Probability},
	number = {3},
	publisher = {Institute of Mathematical Statistics},
	pages = {943 -- 954},
	year = {1986}
}

@article{kolesko2015moment,
	title={Moment estimates for chaoses generated by symmetric random variables with logarithmically convex tails},
	author={Kolesko, Konrad and Lata{\l}a, Rafa{\l}},
	journal={Statistics \& Probability Letters},
	volume={107},
	pages={210--214},
	year={2015},
	publisher={Elsevier}
}

@article{meller2019tail,
	title={Tail and moment estimates for a class of random chaoses of order two},
	author={Meller, Rafa{\l}},
	journal={Studia Mathematica},
	volume={249},
	pages={1--32},
	year={2019},
	publisher={Instytut Matematyczny Polskiej Akademii Nauk}
}

@article{meller2016two,
	title={Two-sided moment estimates for a class of nonnegative chaoses},
	author={Meller, Rafa{\l}},
	journal={Statistics \& Probability Letters},
	volume={119},
	pages={213--219},
	year={2016},
	publisher={Elsevier}
}

@article{polynomials_alpha_subexp,
	author = {Friedrich Götze and Holger Sambale and Arthur Sinulis},
	title = {{Concentration inequalities for polynomials in $\alpha$-sub-exponential random variables}},
	volume = {26},
	journal = {Electronic Journal of Probability},
	number = {none},
	publisher = {Institute of Mathematical Statistics and Bernoulli Society},
	pages = {1 -- 22},
	year = {2021}
}

@article{adamczak2015concentration,
	title={Concentration inequalities for non-Lipschitz functions with bounded derivatives of higher order},
	author={Adamczak, Rados{\l}aw and Wolff, Pawe{\l}},
	journal={Probability Theory and Related Fields},
	volume={162},
	number={3},
	pages={531--586},
	year={2015},
	publisher={Springer}
}

@inbook{vershynin_2012, place={Cambridge}, title={Introduction to the non-asymptotic analysis of random matrices}, DOI={10.1017/CBO9780511794308.006}, booktitle={Compressed Sensing: Theory and Applications}, publisher={Cambridge University Press}, author={Vershynin, Roman}, editor={Eldar, Yonina C. and Kutyniok, GittaEditors}, year={2012}, pages={210–268}}

@unpublished{kronecker_jl_preprint,
	author = "Bamberger, Stefan and Krahmer, Felix and Ward, Rachel",
	title = "Johnson-Lindenstrauss Embeddings with Kronecker Structure",
	note = "Preprint"
}
	
\end{document}